\documentclass[12pt, reqno]{amsart}
\usepackage{hyperref}
\usepackage{graphicx,epsfig}
\usepackage[numbers, comma,sort&compress]{natbib}
\usepackage{fancyhdr,latexsym,amsmath,amsfonts,amssymb,amsbsy,amsthm,url}

\usepackage{euscript}
\allowdisplaybreaks
\newtheorem{theorem}{Theorem}[section]
\newtheorem{lemma}{Lemma}[section]
\newtheorem{proposition}{Proposition}[section]

\theoremstyle{definition}
\newtheorem{definition}{Definition}[section]

\theoremstyle{remark}
\newtheorem{remark}{Remark}[section]

\numberwithin{equation}{section}

\usepackage{euscript}
\newlength{\myfboxsep}
\newlength{\mywidth}
\DeclareMathOperator{\E}{E}

\DeclareMathOperator{\lito}{o}

\DeclareMathOperator{\tr}{tr}

\DeclareMathOperator{\Tr}{Tr}

\newcommand{\comment}[1]{}

\begin{document}
\title[Joint convergence]{Joint convergence of several copies of different patterned random matrices}
\author[R. Basu]{Riddhipratim Basu}
\thanks{R. Basu's research is supported by Lo\`eve Fellowship, Department of Statistics, University of California, Berkeley.}
\address{Department of Statistics
 University of California, Berkeley}
\email{riddhipratim@stat.berkeley.edu}
\author[A.Bose]{Arup Bose}
\thanks{ A. Bose's research is
supported by J.C.Bose National Fellowship, Dept. of Science and Technology, Govt. of India.}
\address{Statistics and Mathematics Unit\\ Indian
Statistical Institute\\ 203 B.T.~Road\\ Kolkata 700108\\ India}
\email{bosearu@gmail.com}
\author[S. Ganguly]{Shirshendu Ganguly}
\address{Department of Mathematics, University of Washington,
Seattle}
\email{sganguly@math.washington.edu}
\author[R. S. Hazra]{Rajat Subhra Hazra}
\address{Institut f\"ur Mathematik\\
Universit\"at  Z\"urich\\
Winterthurerstrasse 190\\
CH-8057, Z\"urich
} \email{rajatmaths@gmail.com}

\begin{abstract}
We study the joint convergence of independent copies of several patterned matrices in the non-commutative probability setup. In particular, joint convergence holds for the well known Wigner, Toeplitz, Hankel, Reverse Circulant and Symmetric Circulant matrices. We also study some properties of the limits. In particular, we show that copies of Wigner becomes asymptotically free with copies of any of the above other matrices.
\end{abstract}
\maketitle

\section{Introduction}\label{sec:intro}
A \textit{non-commutative probability space} is a pair $(\mathcal A, \varphi)$, where $\mathcal A$ is a unital algebra over
 $\mathbb C$ and $\varphi:\mathcal A \to \mathbb C$ is a linear functional such that $\varphi(1) = 1$;
$\varphi$ is
 a \textit{state} if for $a\geq0$ we have $\varphi(a)\geq 0$ and it is
 \textit{tracial} if $\varphi(ab)=\varphi(ba)$ for all $a,b$.
 Elements of $\mathcal A$
 will be called \textit{variables}.

The connection between
large dimensional random matrices (matrices whose elements are random variables) and
non-commutative probability spaces is well known and deep.
Let $(X,\mathcal{B},\mu)$ be a probability
space. Let $L(\mu):=\displaystyle{\bigcap_{p\geq 1}}L^p(X,\mu)$ be
the algebra of random variables with finite moments of all orders.
Set
\begin{equation}\label{eq:matrixalgebra}\mathcal{A}_n:=Mat_n(L(\mu))\end{equation}
as the space of $n\times n$ complex random matrices with
entries coming from $L(\mu)$.
Then $(\mathcal{A}_n, \varphi_j)$, $j=1, 2$ are non-commutative probability spaces where
\begin{equation}\label{varphi}\varphi_1(A)=\frac{1}{n}\Tr(A)\text{ and }\varphi_2(A)=\frac{1}{n}\E[\Tr(A)].\end{equation}
The \textit{joint distribution} of a family $(a_i)_{i \in I}$ of
variables in
$(\mathcal
A,\varphi)$ is the collection of \textit{joint moments} $\{\varphi(a_{i_1}\cdots a_{i_k})\}, \ k \in \mathbb N\ \ \text{and}\ \ i_1,\cdots,i_k \in
 I.$
Let $(\mathcal{A}_n,\varphi_n)_{n\geq 1}$ and
$(\mathcal{A},\varphi)$ be non-commutative probability spaces and let
$(a_{i,n}; i\in I)\subset \mathcal{A}_n$ for each $n$, $(a_{i}; i\in
I)\subset \mathcal{A}$. Then  $(a_{i,n}; i\in I)$ \textit{converges
in
distribution} to $(a_{i}; i\in I)$ if all joint moments converge. Equivalently, for all $ p\in
\mathbb{C}[X_i,i\in I] $,
\begin{equation}\label{eq:cid}
\lim_n\varphi_n(p(\{a_{i,n}\}_{i\in I}))= \varphi(p(\{a_{i}\}_{i\in I})).
\end{equation}
Convergence
of  an $n \times n$  real symmetric matrix $A_n$
with respect to $\varphi_1$ and  $\varphi_2$ demands convergence  for each non-negative integer $k$, respectively of $\varphi_1(A_n^k)$ (almost surely) and
$\varphi_2(A_n^k)$.

A related notion of convergence
is that of the spectral distribution.
If the eigenvalues of $A_n$ are $\{\lambda_i\}$, then the spectral measure of $A_n$ is defined as
\begin{equation}\label{def:esd}
L_n=\frac1n\sum_{i=1}^n\delta_{\lambda_i}.\end{equation}
If as $n \to \infty$, $L_n$ converges weakly
(almost surely) to a measure $\mu$ with distribution function $F$ say, then $F$ (or $\mu$) is called the \textit{limiting spectral distribution} (LSD) of
$\{A_n\}$.

In his
pioneering work,
\citet{wigner1958} showed
that
the GUE (Gaussian Unitary Ensemble, Hermitian matrices with
i.i.d.\ complex Gaussian entries with variance $1/n$)
converges with respect to $\varphi_2$
to the \textit{semi-circular} variable $s$ characterized by the limit moments
$$\varphi(s^k)=\int t^k\frac1{2\pi}\sqrt{4-t^2}\boldsymbol{1}_{|t|\le 2}dt.$$
The probability law with density $\displaystyle{\sqrt{4-t^2}\boldsymbol{1}_{|t|\le 2}}$ having the above moments is called
the \textit{semi-circle law}.
 This result was extended in many directions
 for Gaussian Orthogonal Ensemble (GOE) and Gaussian Symplectic Ensemble (GSE) and in fact for i.i.d.\ entries with finite second moment. See \citet{Baisilverstein} for detailed treatment.


 \citet{voiculescu1986addition} introduced
the notion of freeness in the context of free groups. It
 played the role of independence in non-commutative
probability spaces.
Unital subalgebras $\{\mathcal{A}_i\}_{i\in I}\subset\mathcal{A}$ are said to be
{\it free} if $\varphi(a_1\cdots a_n)=0$ whenever $\varphi(a_j)=0,\  a_j\in \mathcal{A}_{i_j}$ and $i_j\neq i_{j+1}$ for all $j$.
%

The  notions of freeness and of convergence as in \eqref{eq:cid} together yield an obvious and  natural notion of \textit{asymptotically free}.
 \citet{voi} showed that if we take $k$ independent Hermitian
random matrices $\{W_{i,n}\}_{1\leq i\le k}$ distributed as GUE  then they are asymptotically free. 
In other words, for any polynomial $\mathbb P$ in $k$ variables,
\begin{equation*}
 \E\left[\frac1n \Tr (\mathbb P(W_{1,n}, \ldots, W_{k,n}))\right]\to \tau(\mathbb P(s_1, \ldots, s_k)) \text{ as } n\to\infty,
\end{equation*}
where $(s_1,\ldots, s_k)$ is a collection of free (and semi-circular)  variables in some non-commutative probability space $(A,\tau)$.
 Asymptotic freeness  of GUE has been a key feature in the development of free probability and its various applications.
\citet{voi} also showed the asymptotic freeness of
GUE and diagonal constant matrices.
 Later,  \citet{voiculescu:1998} improved the result to asymptotic freeness of
 GUE and general $n \times n$ deterministic matrices $\{D_{i,n}\}$ (having LSD) and satisfying
\begin{equation}\label{eq:cond:free}
 \sup_n\|D_{i,n}\|<\infty \text{ for each $i$,}
\end{equation}
where $\|\cdot\|$ denotes the operator norm.
This
inclusion of constant
matrices
had important implications in the factor theory of von Neumann algebras.
\citet{dykema:1993}
established
a similar result
for a family of independent  Wigner matrices (symmetric matrix with i.i.d.\ real entries with uniformly bounded moments) and block-diagonal constant matrices with bounded block size. The results were also shown to
hold
with respect to $\varphi_1$
almost surely
(see \citet{hiai:petz:2000, hiai:petz:book} for details). For general results on
freeness between Wigner and deterministic matrices we refer to \citet{anderson:guionnet:zeitouni}.
Various other extensions to Wishart ensembles, GOE, GSE are also available.
See \citet{voiculescu:1998, Capitaine:Casalis:2004,Capitaine:Martin:2007, collins:guionnet:segala:2009,Schultz:2005, ryan:1998}.

Freeness is present  elsewhere too
and
one
important place
is
the Haar distributed matrices.
It is well known that any unitary invariant matrix (in particular GUE) can be written as $UDU^*$ where
$D$ is a diagonal matrix and $U$ is Haar distributed on the space of unitary matrices
and independent of $D$.
\citet{voi} showed that
$\{U,U^*\}$ and $D$ are asymptotically  free.
\citet{hiai:petz:2000} showed that the Haar unitaries and general deterministic matrices satisfying~\ref{eq:cond:free} are almost surely asymptotically free.
\citet{Collins} showed that general deterministic matrices and Haar measure on unitary group are asymptotically free
almost surely provided the deterministic matrices jointly converge. The case for orthogonal and symplectic groups were dealt with in \citet{collins:sniady:2006}.

One of the important applications of these in random matrix theory was the study of the spectrum
of  $W_n+P_n$ where $W_n$ is a Wigner matrix and $P_n$ is another suitable matrix. The spectrum of this perturbation
has been of interest for a long time (see \citet{fulton}). Suppose the spectral
measure of $P_n$ weakly converges to $\mu_P$. Then
the spectral measure of $W_n+P_n$
converges weakly and almost surely and in expectation to the
\textit{free convolution} of $\mu_P$ and the semicircular law of whenever
  $\mu_P$ has compact support or $P_n$ satisfy
  \ref{eq:cond:free}.
These results were derived using asymptotic freeness results between
deterministic (or random) matrices and Wigner matrix. \citet{pastur:vasilchuk} extended these results for unbounded perturbations (possibly random) using analytic machinery of Stieltjes transform. It is to be noted that this  result on the  sum
does not yield asymptotic freeness
between the
matrices.

The special case where $P_n$ has finite rank has  received considerable amount of
interest recently.
In this case, the limit measure is still the semi-circular law but the
behavior at the edge has some interesting properties.
See \citet{peche:2007,CDF:2009, BGM:2011, CDF:2011, peche}.

One relevant question is whether this asymptotic freeness persists
for some
other types of matrices. Consider the
class of \textit{patterned matrices}.
These are matrices where, along with symmetry, some other assumptions are imposed on the structure.
 Important examples
 are the  Toeplitz,
Hankel, Symmetric Circulant and Reverse Circulant. The spectrum of these matrices were studied in \citet{bry, bosesen, hammond:miller}.
Generally speaking
the Stieltjes transform does not seem to be a convenient tool to study these matrices due to the strong dependence among the rows and columns.
\citet{bosehazrasaha} showed that under suitable assumptions on the pattern,
there is joint convergence of i.i.d.\ copies of a \textit{single} pattern matrix as dimension goes to infinity. One
important
consequence is that in the limit other kinds  of non-free independence may arise. In particular,
Symmetric Circulants
are commutative
and Reverse Circulants are
asymptotically half independent.
As yet, no description of independence is available for the Toeplitz and Hankel matrices.

As a more general goal,
we investigate the  joint convergence of multiple independent copies of these matrices, including the Wigner.
Inter alia,
 we address the asymptotic freeness of the Wigner matrices and patterned matrices.

In Theorem~\ref{main theorem-1}, we provide
sufficient conditions for
joint convergence holds.
We deal with only real symmetric matrices as the structure of many of these matrices change if one takes  complex entries.
One of the basic \textit{necessary} assumptions  on the pattern matrices is \textit{Property B}, which states that the maximum number of times any entry is repeated in a row remains uniformly bounded across all rows as $n\to\infty$. All the above five matrices 
 satisfy
 Property B. Under Property B and some moment assumptions on the entries we show that if a criteria (Condition~\ref{psi}) holds for one copy each of any subcollection of matrices, then the joint convergence holds for multiple copies. This Condition~\ref{psi} is
 satisfied by all the five matrices.
  We use the method of moments and the so called volume method to prove these results.
  See~\citet{bosesen, bry} for the use of volume method for convergence of spectral measure of patterned matrices.
 As an application of
 Theorem~\ref{main theorem-1}, the following holds: if $\mathbb P$ is a symmetric polynomial in any of the two following scaled matrices: Wigner, Toeplitz, Hankel, Reverse Circulant and Symmetric Circulant with uniformly bounded entries then the spectral measure $L_n$ of the matrix $\mathbb P$ converges to a non-random measure $\mu$  on $\mathbb R$ weakly almost surely.


  In Theorem \ref{theorem:free1}, we show that any collection of Wigner matrices is free of the other four matrices. As already discussed,
 Wigner and deterministic matrices are asymptotically free.
 By the results of
 \citet{Collins} and \citet{collins:sniady:2006} the results are true for general deterministic matrices which converge jointly. To the best of our knowledge these results directly do not imply
 the freeness result Theorem~\ref{theorem:free1}.
  This is because, the existing results
 need some conditions on the behavior of the trace of the matrices as pointed out in Remark 3.6 of \citet{Collins}. The condition in \citet{Collins} (equation (3.4) therein) was studied in \citet{Capitaine:Casalis:2004}. It was shown that under the technical condition on the random matrices (see Condition $C$ and $C'$ in \citet{Capitaine:Casalis:2004}) there is asymptotic freeness
  between Wigner and other random matrices. Although the Theorems of \citet{Capitaine:Casalis:2004} are for GUE,
  it is expected that the results would be true for real entries or GOE.   In other available criteria for freeness, condition  (\ref{eq:cond:free}) appears
(see \citet{anderson:guionnet:zeitouni} and Theorem 22.2.4 of \citet{speicher:2011}). This is
not applicable in our situation as it is known from the works of \citet{bosesen, bry} that the spectral norm of Toeplitz, Hankel, Reverse Circulant and  Symmetric Circulant are unbounded.

 Instead of attempting to check/modify  the technical sufficient condition  of \citet{Capitaine:Casalis:2004}  we
extend
the volume method to derive Theorem~\ref{theorem:free1}.
This technique
is
similar in spirit to those in Chapter 22 of \citet{spei}. However, we bypass  the detailed  properties of permutation group and Weingarten functions.
 It is quite feasible that the techniques of \citet{Collins} and \citet{Capitaine:Casalis:2004} may be extended to
 prove Theorem~\ref{theorem:free1}.
 Incidentally,
 if we take the Wigner with complex entries then
 Theorem~\ref{theorem:free1} holds for any patterned matrix satisfying Property B and having an LSD.

The use of random matrix theory and free probability in CDMA (Code Division Multiple Access) and MIMO
(multiple input and multiple
output)  systems was shown in many articles.
See \citet{oraby2, Oraby, tulino:verdu, jack:mimo}.
For a MIMO system with $n_1$ transmitter antenna and $n_2$ receiver antenna, the
received signal
is
represented in terms of equation $\mathbb Y_n= \mathbb H \mathbb A_n+\mathbb
B_n$ where
$\mathbb A_n$ is an $n_1$- dimensional vector depending on $n$ and $\mathbb B_n$
is a noise signal and $\mathbb H$
is the channel matrix which generally has a block structure as below and
$\mathbb Y_n$ is an $n_2$ dimensional vector.
$$ \mathbb H= \left[ \begin{array}{ccccccccc}
		C_1 & C_2 & \hdots & C_L & \mathbf 0 & \hdots & &\hdots &
\mathbf 0 \\
		\mathbf 0 & C_1 & C_2 & \hdots & C_L & \mathbf 0 & & & \vdots\\
		\vdots & \mathbf 0 & C_1 & C_2 & \hdots & C_L & \mathbf 0 && \\
		 &  & \ddots & \ddots & \ddots & & \ddots &  \ddots &\vdots\\
		\vdots &	  & & \ddots & \ddots & \ddots & &\ddots & 	
\mathbf 0									
\\
		\mathbf 0 &  \hdots &  & \hdots &\mathbf 0 & C_1 & C_2 &\hdots
& C_L \end{array} \right].$$
 One of the main issues in the study
 of a MIMO system is
 the eigenvalue distribution
of $\mathbb H\mathbb H^*$ since this is linked to
the capacity of the channel.
Here $\{C_i\}$ can be Wigner matrices or more general matrices. It may also
happen that some of the blocks are
Toeplitz or Hankel or any other structured matrices.  Studying the  spectral
properties of such matrices boils down to studying
the joint convergence of different patterned matrices. The results of this article can be used for studying such systems. We refer
the readers
to the recent article by \citet{male} which applies similar results for MIMO system.

Finally we
point out that we could not obtain full characterization
 of the joint limits if one of the matrices is not Wigner.
 It is known that in a complex unital algebra only two notions of independence of subalgebras may arise:
 freeness and classical independence (see \citet{speicher1997}).  Although Reverse Circulant limit shows half independence,
this notion
is only for variables
of an algebra and not for subalgebras
(see \citet{bhs:half}).
 For other matrices like Toeplitz and Hankel nothing is known yet about the joint convergence.

In
Section~\ref{sec:def} we
recall definitions of pattern matrices and express the trace in terms of circuits and words (equivalently pair-partitions). In Section~\ref{main results} we state our main results on
joint convergence of patterned matrices including those mentioned earlier as well as
Theorem~\ref{theorem:catalan} on the contribution of  certain monomials depending on the structure of the matrices.
We also discuss the properties of the sum of two random matrices in the limit. The final Section~\ref{sec:proof} is dedicated to the proofs.

\section{Some basic definitions and notation}\label{sec:def}
\subsection{Patterned matrices, link function, trace formula and words}
Patterned matrices are defined via the \textit{link functions}.  A link function $L$ is defined as a
function $L:\{1,2,..,n\}^2 \rightarrow Z_{\ge}^d, n \ge 1$. For our purposes $d=1$ or $2$. Although $L$ depends on $n$, to avoid complexity of notation we suppress
the $n$ and consider $\mathbb N^2$ as the common domain. We also assume that $L$ is symmetric in its arguments, that is, $L(i,j)=L(j,i)$.

Let $\{x(i)\}$ and $\{x({i,j})\}$ be a sequence of real random
variables, referred to as the \textit{input sequence}. The sequence
of matrices $\{A_n\}$ under consideration  will be defined by
$$A_n \equiv ((a_{i,j}))_{1\leq i, j,\leq n} \equiv ((x{(L(i,j))})).$$ Some important matrices we shall discuss in this article are:
\begin{itemize}
\item[($W_n$)] Wigner matrix: $ L:\mathbb{N}^2 \rightarrow \mathbb Z^2$ where
$L(i,j) = (\text{min}(i,j) , \text{max}(i,j)).$
\vskip5pt
\item[($T_n$)] Toeplitz matrix: $ L:\mathbb{N}^2 \rightarrow \mathbb Z $ where
$L(i,j) =|i-j|.$
\vskip5pt
\item[($H_n$)] Hankel matrix:  $L:\mathbb{N}^2 \rightarrow \mathbb Z$ where
$L(i,j) =i+j.$
\vskip5pt
\item[($RC_n$)] Reverse Circulant: $L:\mathbb{N}^2 \rightarrow \mathbb Z $ where
$L(i,j) =(i+j )\mod n.$
\vskip5pt
\item[($SC_n$)] Symmetric Circulant: $L:\mathbb{N}^2 \rightarrow \mathbb Z $ where
$L(i,j) =n/2-|n/2-|i-j||.$
\end{itemize}
It is now well known that the limiting spectral
distribution (LSD) of the above matrices exists.
\citet{bosehazrasaha} reviewed the results on LSD of the above
matrices. For various results on Wigner matrices we refer to the
excellent exposition by \citet{anderson:guionnet:zeitouni}.

The $L$ function for all the five matrices defined above satisfy the
following  property. This property was introduced by \citet{bosesen}
and shall be crucial to us.
(For any set $S$, $\#S$ or $|S|$ will denote the number of elements in $S$).\vskip5pt

\textit{Property B}: We say a link function $L$ satisfies \textit{Property B} if,
\begin{equation}\label{prop:B}
\Delta(L) = \sup_n \sup_{ t \in \mathbb Z^d_{\ge} } \sup_{1 \le k \le n} \# \{
l: 1 \le l \le n, L(k,l) = t\} < \infty.
\end{equation}
In particular, $\Delta(L)=2$ for $T_n,SC_n$ and $\Delta(L)=1$ for $W_n,H_n$ and $RC_n$.
\vskip5pt

Consider $h$ different type of patterned matrices where type $j$
has $p_j$ independent copies, $1 \leq j \leq h$.  The different link functions shall
be referred to as \textit{colors} and different independent copies
of the matrices of any given color shall be referred to as
\textit{indices}. Let $\{X_{i,n}^j, 1\leq i\leq p_j\}$
$n\times n$ symmetric patterned matrices with link
functions $L_j$, $j=1,2,\cdots,h$. Let
$X_i^j(L_j(p,q))$
denote the $(p,q)$-th
entry of $X_{i,n}^j$.
We suppress the
dependence on $n$ to simplify notation.
Two
natural assumptions on the link function and the input sequence are:\vskip5pt

\let\myenumi\theenumi
\renewcommand{\theenumi}{A\myenumi}
\begin{enumerate}
 \item All link functions $\{L_j,j=1,2,\cdots,h\}$ satisfy \textit{Property B}, that is,
$$\max_{1\leq j\leq h}~\sup_{n\geq 1}~\sup_t\sup_{1\leq p \leq n}~ \#\{q:1\leq q \leq n, L_j(p,q)=t\}\leq \Delta<\infty.$$\label{A1}
\item Input sequences $\{X_i^j(k): k\in\mathbb Z \text{ or }\mathbb Z^2\}$ are real random variables independent across $i$, $j$ and $k$ with mean zero and variance $1$ and the moments are uniformly bounded, that is,
$$\sup_{1\leq j\leq h}~\sup_{1\leq i\leq p_j}~\sup_{n\geq 1}~\sup_t~\sup_{1\leq p,q\leq n}\E\left[|X_{i}^j(L_j(p,q))|^k\right]\leq c_k<\infty.$$\label{A2}
\end{enumerate}
 \let\theenumi\myenumi

We consider $\{\frac1{\sqrt n}X_{i,n}^j, 1\leq i\leq p_j\}_{1\leq j\leq h}$ as elements of $\mathcal{A}_n$  given in (\ref{eq:matrixalgebra}) and investigate the joint convergence with respect to the normalized tracial states $\varphi_1$ or $\varphi_2$ (as in \eqref{varphi}). The sequence of matrices jointly converge if and only if for all monomials $q$, $$\varphi_d\left(q\left(\frac1{\sqrt n}\{X_{i,n}^j, 1\leq i\leq p_j\}_{1\leq j\leq h}\right)\right)$$ converge to a limit as $n\to\infty$ for either $d=1$ or $d=2$. For $d=1$, the convergence is in the almost sure sense. The case of $h=1$ and $p_1 = 1$ (a single patterned matrix) was dealt in \citet{bosesen} and $h=1$ and $p_1 > 1$ (i.i.d. copies of a single patterned matrix) was dealt in \citet{bosehazrasaha}. In particular, convergence holds for i.i.d. copies of any one of the five patterned matrices. The starting point in showing this was the trace formula. The related concepts of circuits, matchings and words  will be extended below to multiple copies of several matrices.

Since our primary aim is to show convergence for every monomial, we shall from now on, fix an arbitrary  monomial $q$ of length $k$. Then we may write,
\begin{equation}
 \label{eq:monomial}
q\left(\frac1{\sqrt n}\{X_{i,n}^j, 1\leq i\leq p_j\}_{1\leq j\leq h}\right)=\frac1{n^{k/2}}Z_{c_1,t_1}Z_{c_2,t_2}\cdots Z_{c_k,t_k},
\end{equation}
where $Z_{c_m,t_m}=X_{t_m}^{c_m}$ for $1\leq m\leq n$.

From~\eqref{eq:monomial} we get,
\begin{align}\label{eq:zeepi}
\widetilde{\mu_n}(q)&:=\frac1n\Tr\left[\frac1{n^{k/2}}Z_{c_1,t_1}Z_{c_2,t_2}\cdots Z_{c_k,t_k}\right]\notag\\
&=\frac{1}{n^{1+k/2}}\sum_{j_1,j_2,\cdots,j_k}\left[Z_{c_1,t_1}(L_{c_1}(j_1,j_2))Z_{c_2,t_2}(L_{c_2}(j_2,j_3))\cdots
Z_{c_k,t_k}(L_{c_k}(j_k,j_1))
\right] \notag \\
&=\frac{1}{n^{1+k/2}}\sum_{\substack{\pi:\{1,\cdots,k\}\to \{1,\cdots,n\}\\ \pi(0)=\pi(k)}}\prod_{i=1}^k Z_{c_i,t_i}(L_{c_i}(\pi(i-1),\pi(i)))\notag \\
&=\frac{1}{n^{1+k/2}}\sum_{\substack{\pi:\{1,\cdots,k\}\to \{1,\cdots,n\}\\ \pi(0)=\pi(k)}}\mathbf{Z}_\pi \quad \text{say}.
\end{align}
Also define,
\begin{equation}\label{eq:zeepialmostsure}\widehat{\mu_n}=\E[\widetilde{\mu_n}].\end{equation}
Keeping in mind that we seek to show the existence of the limits in
(\ref{eq:zeepi}) and (\ref{eq:zeepialmostsure}) as $ n \to \infty$, we
now develop some appropriate notions. In particular these help us to show that certain terms in
these sums are negligible in the limit.

Any map $\pi:\{1,\cdots,k\}\to \{1,\cdots,n\}$ with $\pi(0)=\pi(k)$
will be called a \textit{circuit}. Its dependence on $k$ and
$n$ will be suppressed. Observe that
$\widetilde{\mu_n}$ and $\widehat{\mu_n}$ involve sums over
circuits. Any value $L_{c_i}(\pi(i-1),\pi(i))$ is called an
\textit{$L$-value} of $\pi$. If an $L$-value is repeated $e$ times in $\pi$ then $\pi$ is said to have an \textit{edge} of order $e$. Due to independence and mean zero of
the input sequences,
\begin{equation}\label{eq:nonmatch}\E[\mathbf{Z}_\pi]=0 \quad\text{if $\pi$ has any edge of order one.}\end{equation}
If all $L$-values appear more than once then we say the circuit is
\textit{matched} and only these circuits are relevant due to the
above.

A circuit is said to be
\textit{color matched} if all the $L$-values are repeated within the
same color. A circuit is said to be \textit{color} and \textit{index} matched if in
addition, all the $L$-values are also repeated within the same
index.

Denote the colors and indices  present in
$q$ by $(c_1,c_2,\cdots, c_k)$ and $(t_1,t_2,\cdots,t_k)$ respectively.
We can define an equivalence relation on the set of color and index
matched circuits, extending the ideas of \citet{bosehazrasaha} and
\citet{bosesen}. We say $\pi_1\sim \pi_2$ if and only if their
matches take place at the same colors and at the same indices. Or,
\begin{align*}
 c_i=c_j, t_i=t_j \ \ \text{and}\ \ L_{c_i}(\pi_1(i-1),\pi_1(i))&=L_{c_j}(\pi_1(j-1),\pi_1(j))\\
 \Longleftrightarrow &&\\
c_i=c_j, t_i=t_j\ \ \text{and}\ \ L_{c_i}(\pi_2(i-1),\pi_2(i)))&= L_{c_j}(\pi_2(j-1),\pi_2(j)).
\end{align*}
An equivalence class can be expressed as a colored and indexed word
$w$: each word is  a string of letters in alphabetic order of their first occurrence with a
subscript and a superscript to distinguish the index and  the color respectively. The $i$-th position of $w$ is denoted by
$w[i]$. Any $i$ is a \textit{vertex} and it is \textit{generating} (or \textit{independent}) if either $i=0$ or $w[i]$ is the position of the first occurrence of a letter. By abuse of notation we also use $\pi(i)$ to denote a vertex.

For example, if $$q=X_1^1X_2^1X_1^2X_1^2X_2^2X_2^2X_2^1X_1^1=
Z_{1,1}Z_{1,2}Z_{2,1}Z_{2,1}Z_{2,2}Z_{2,2}Z_{1,2}Z_{1,1},$$ then
$a_1^1b_2^1c_1^2c_1^2d_2^2d_2^2b_2^1a_1^1$ is \textit{one}
colored and
indexed word corresponding to $q$. Any colored and indexed word uniquely determines the monomial it
corresponds to. A colored and indexed (matched) word is
\textit{pair-matched} if all its letters  appear exactly
twice. We shall see later that under \textit{Property B}, only such circuits
and words survive in the limits of (\ref{eq:zeepi}) and
(\ref{eq:zeepialmostsure}).

Now we define some useful subsets of the circuits. For a colored and indexed word $w$, let
\begin{equation}\label{eq:pici}\Pi_{CI}(w)=\{\pi : w[i]=w[j] \Leftrightarrow  (c_i,t_i,L_{c_i}(\pi (i-1), \pi(i)))=(c_j,t_j, L_{c_j}(\pi(j-1), \pi(j))\}.\end{equation}
Also define
\begin{equation}\label{eq:picistar}
\Pi_{CI}^*(w)=\{\pi : w[i]=w[j]\Rightarrow (c_i,t_i, L_{c_i}(\pi (i-1), \pi(i))= (c_j,t_j, L_{c_j}(\pi (j-1), \pi(j))\}.\end{equation}
Every colored and indexed word has a corresponding non-indexed
version which is
obtained by dropping the indices from the letters
(i.e. the subscripts). For example, $a_1^1b_2^1c_1^2c_1^2d_2^2d_2^2b_2^1a_1^1$ yields $a^1b^1c^2c^2d^2d^2b^1a^1$. For any monomial $q$,
dropping the indices amounts to replacing, for every $j$,  the independent copies
$X_i^j$ by a single $X^j$ with link function $L_j$.
In other words it corresponds to the case where
$p_j=1$ for $1\leq j\leq h$.

Let $\psi (q)$ be the monomial obtained by dropping the
indices from $q$. For example, $$\text{if}\ \
q=Z_{1,1}Z_{1,2}Z_{2,1}Z_{2,1}Z_{2,2}Z_{2,2}Z_{1,2}Z_{1,1}\ \
\text{then}\ \ \psi(q)=Z_1Z_1Z_2Z_2Z_2Z_2Z_1Z_1.$$
(\ref{eq:pici}) and (\ref{eq:picistar}) get mapped to the following subsets of  non-indexed colored
word $w'$ via $\psi$:
$$\Pi_{C}(w)=\{\pi : w[i]=w[j] \Leftrightarrow  c_i=c_j~ \text{and}~ L_{c_i}(\pi (i-1), \pi(i))=L_{c_j}(\pi(j-1), \pi(j))\},$$
$$\Pi_{C}^*(w)=\{\pi : w[i]=w[j] \Rightarrow  c_i=c_j~ \text{and}~ L_{c_i}(\pi (i-1), \pi(i))=L_{c_j}(\pi(j-1), \pi(j))\}. $$
Since pair-matched words are going to be crucial, let us define:
\begin{align*}
 CIW(2)&=\{w: w \ \ \text{is indexed and colored pair-matched corresponding to $q$}\}\\
 CW(2) &=\{w:w \ \ \text{is non-indexed colored pair-matched corresponding to $\psi(q)$}\}.
\end{align*}
 For $w \in CIW(2)$, let us consider the word obtained by dropping the indices of $w$. This defines an injective mapping into $CW(2)$ and we
 continue to denote this mapping by $\psi$.

For any $w \in CW(2)$ and $w'\in CIW(2)$, we define (\textit{whenever the limits exist}),
$$p_C(w)=\lim_{n\rightarrow \infty}\frac{1}{n^{1+k/2}}|\Pi^*_C(w)|\quad\text{and}\quad p_{CI}(w')=\lim_{n\rightarrow\infty}\frac1{n^{1+k/2}}|\Pi^*_{CI}(w')|.$$
\section{Main results}\label{main results}
Our first result is on the joint convergence of several patterned random
matrices
and is analogous to
Proposition 1 of \citet{bosehazrasaha} who considered the case $h=1$.
\begin{theorem}\label{main theorem-1}
Let $\{\frac1{\sqrt n}X_{i,n}^j, 1\leq i\leq p_j\}_{1\leq j\leq h}$ be a sequence of real symmetric patterned random matrices satisfying Assumptions~\eqref{A1} and~\eqref{A2}.
Fix a monomial $q$ of length $k$ and assume that, for all $w\in CW(2)$
\begin{equation}
\label{psi} p_C(w)=\lim_{n\rightarrow
\infty}\frac{1}{n^{1+k/2}}|\Pi^*_C(w)|~~~\text{exists.}
\end{equation}
Then,
\begin{enumerate}
\item \label{main theo-1:claim 1}for all $w\in CIW(2)$, $p_{CI}(w)$ exists  and $p_{CI}(w)=p_C(\psi(w))$,
\vskip5pt

\item \label{main theo-1:claim 2} we have
\begin{equation}
\lim_{n\rightarrow \infty}\widehat{\mu}_n(q)=\sum_{w\in CIW(2)}p_{CI}(w)=\alpha (q) \text{(say)}
\end{equation}
with
$$|\alpha(q)| \leq
\begin{cases}
 \frac{k!\Delta^{k/2}}{(k/2)!2^{k/2}}& \text{if $k$ is even and each index appears even number of times}\\
  0 &\text{otherwise}.
\end{cases}$$
\item \label{main theo-1:claim 3} $\lim_{n\rightarrow \infty}\widetilde{\mu_n}(q)=\alpha(q)$ almost surely.
\end{enumerate}
As a consequence if  \eqref{psi} holds for every $q$ then  $\{\frac1{\sqrt n}X_{i,n}^j, 1\leq i\leq p_j\}_{1\leq j\leq h}$ converges jointly in both the states $\varphi_1$ and $\varphi_2$  and the limit is independent of the input sequence.
\end{theorem}
\begin{remark}\label{remark:sum}
(i) Theorem \ref{main theorem-1} asserts that if the joint
convergence holds for $p_j=1, j=1,2,\cdots,h$ (that is if condition
(\ref{psi}) holds), then the joint convergence continues to  hold
for
$p_j\geq 1$. There is no general way of checking
\eqref{psi}. However,  see the next theorem.
\vskip5pt
(ii)
Under the conditions of Theorem \ref{main theorem-1}, for any fixed monomial $q$ that yields a symmetric matrix, the corresponding LSD exists.  Using truncation arguments, it is possible to prove this under the weaker assumption  that
the  input sequence
is i.i.d.\ with second moment finite.
\end{remark}
\begin{theorem}\label{main theorem-2}
Suppose Assumption~\eqref{A2} holds. Then  $p_C(w)$ exists for all monomials $q$ and for all $w\in CW(2)$, for any two of the following matrices at a time:
Wigner, Toeplitz, Hankel, Symmetric Circulant and Reverse Circulant.
\end{theorem}

Theorem~\ref{main theorem-1} and Theorem~\ref{main theorem-2} shows that if $\mathbb P$ is a symmetric polynomial in any of the two matrices Wigner, Toeplitz, Hankel, Symmetric Circulant and Reverse Circulant then the spectral measure of $\mathbb P$ converges almost surely.

In general the value of $p_C(w)$  cannot be computed for arbitrary
pair-matched word. In the two tables,
we provide some examples.
\begin{table}[h!]
{\footnotesize \begin{center}
\caption{$p_C(w)$ for colored words corresponding to monomials $q=q(T,H)$  }
\begin{tabular}{|c|c|c|}\hline
Monomial & Word & $p_C(w)$ \\
\hline
TTHH & aabb & 1 \\
\hline
THTH & abab & 2/3 \\
\hline
TTTTHH & aabbcc & 1 \\
& abbacc & 1\\
& ababcc & 2/3 \\
\hline
HHHHTT & aabbcc & 1\\
& abbacc & 1 \\
& ababcc & 0 \\
\hline
TTHTTH & aabccb & 1 \\
& abcbac & 1/2 \\
& abcabc & 1/2 \\
\hline
HHTHHT & aabccb & 1 \\
& abcbac & 1/2 \\
& abcabc & 0 \\
\hline
\end{tabular}
\end{center}}
\end{table}
\begin{table}[h!]
{\footnotesize \begin{center}
\caption{$p_C(w)$ for colored words corresponding to monomials $q=q(H,R)$ and $q(H,S)$}
\begin{tabular}{|c|c|c||c|c|c| }\hline
Monomial & Word & $p_C(w)$ & Monomial & Word & $p_C(w)$\\
\hline
RRHH & aabb & 1 &SSHH & aabb & 1 \\
\hline
RHRH & abab & 0 & SHSH & abab & 2/3 \\
\hline
RRRRHH & aabbcc & 1 &SSSSHH & aabbcc & 1 \\
& abbacc & 1 && abbacc & 1\\
& ababcc & 0 && ababcc & 1 \\

\hline
HHHHRR & aabbcc & 1 &HHHHSS & aabbcc & 1\\
& abbacc & 1 && abbacc & 1 \\
& ababcc & 0 & &ababcc & 0 \\
\hline
RRHRRH & aabccb & 1 & HHHSHS & aabcbc & 1/2 \\
& abcbac & 0 & & abbcac & 1/2 \\
& abcabc & 2/3 & & abcabc & 0 \\
\hline
HHRHHR & aabccb & 1 &HHSHHS & aabccb & 1 \\
& abcbac & 0 & & abcbac & 1/2 \\
& abcabc & 1/2 & & abcabc & 0 \\
\hline
\end{tabular}
\end{center}}
\end{table}
As seen
in the two tables,
$p_C(w)$ equals one for certain words. We now identify a class of such words.
This has ramifications later in the study of freeness.

If for a $w \in CW(2)$,  sequentially deleting all double letters of
the same color each time  leads to the empty word then we call $w$ a
\emph{colored Catalan word}.

In the non-colored and non-indexed situation, \citet{bosesen} established that $p(w)=1$ for the five
matrices for all Catalan words $w$. \citet{bana:bose}
introduced the following condition which guarantees this.

Consider the following boundedness property of
the
number of matches between rows across all pairs of columns. \vskip5pt

\textit{Property P:}
A link function $L$ satisfies \textit{Property P} if
\begin{equation}\label{eq:propertyp} M^*=\sup_n \sup_{i,j}\#
\{1\leq k\leq n: L(k,i)=L(k,j)\}<\infty.
\end{equation} Note that the five matrices
satisfy \textit{Property P}.\vskip5pt

It is not hard to see that colored Catalan words are in one one correspondence with non-crossing colored pair-partitions. Thus freeness and semi-circularity may be described for our limits in the language of words: if the limit satisfies $p_C(w)=0$ for all words which are not colored Catalan, then the limit is free. {\it In addition}, if $p_C(w)=1$ for all colored Catalan words, then the limits are also semicircular, which is precisely what happens for Wigner matrices. For the other four matrices, the limit is neither semicircular nor free but $p_C(w)=1$ for all colored Catalan words as Theorem \ref{theorem:catalan} shows. This extends the main result of \citet{bana:bose} to multiple copies of colored matrices.
\begin{theorem}\label{theorem:catalan}
(i) Suppose $X$ and $Y$ satisfy Assumption~\eqref{A1} and Assumption~\eqref{A2}. Consider any  monomial in $X$ and $Y$ of length $2k$.
 Then
 $$|\Pi_C^*(w)|\geq n^{1+k} \ \ \text{for any colored Catalan word}\ \  w.$$
 As a consequence, $p_C(w)\geq 1$ for any  colored Catalan word $w$.
\vskip5pt
\noindent
(ii) Suppose the link functions satisfy \textit{Property B} and \textit{Property P} and the input satisfies  Assumption~\eqref{A2}.
 Then for any colored
 Catalan word, $p_C(w)=1$.
\end{theorem}
It is well known that independent Wigner matrices are asymptotically free and also
they are asymptotically free of any class of deterministic matrices $\{D_{i,n}\}_{1\le i\le p}$ which
satisfy \eqref{eq:cond:free} (see Theorem 5.4.5 of \citet{anderson:guionnet:zeitouni}).
Moreover,
the deterministic matrices can be replaced by random matrices $\{A_n\}$ which
$\sup_n\|A_n\|<\infty$ (see \citet{speicher:2011}) or which
satisfy the sufficient condition (Condition C) of \citet{Capitaine:Casalis:2004}.

These
results
cannot be used here
since
the spectral norm of Toeplitz, Hankel, Reverse Circulant and Symmetric Circulant are
unbounded as $n \to \infty$.
Nevertheless, using
the notions of circuits and words we are able to show freeness in
a relatively simple way.
\begin{theorem}
\label{theorem:free1}
Suppose $\{W_{i,n}, 1\leq i\leq p, A_{i,n}, 1\leq i\leq p\}$ are
independent matrices satisfying assumptions~\eqref{A2} where $W_{i,n}$ are Wigner matrices and
$A_{i,n}$ are any of Toeplitz, Hankel, Symmetric Circulant or
Reverse Circulant matrices. Then $\{W_{i,n}, 1\leq i\leq p\}$ and $\{ A_{i,n}, 1\leq i\leq p\}$ are free in the limit.
\end{theorem}
\begin{remark}\label{rem:gue:free}
Incidentally, the freeness between GUE and other patterned matrices is much easier to establish.
Indeed,
 it can be shown that GUE and any patterned matrices having Property B and satisfying~\eqref{A2}, having LSD are
 asymptotically free.
 We  provide a brief proof 
 of this assertion at the end of Section~\ref{sec:proof}.
\end{remark}
\subsection{Sum of patterned random matrices}
\begin{proposition}
\label{prop:sum}
Let $A$ and $B$ be two independent patterned matrices
satisfying Assumptions~\eqref{A1} and~\eqref{A2}. Suppose
$p_C(w)$ exists for every $q$ and every $w$. Then
LSD for $\frac{A+B}{\sqrt{n}}$ exists in the almost sure sense, is symmetric and does not depend on the underlying distribution of the input sequences of $A$ and $B$.
Moreover, if either LSD of $\frac{A}{\sqrt{n}}$ or LSD of $\frac{B}{\sqrt{n}}$ has unbounded support then LSD of $\frac{A+B}{\sqrt{n}}$ also has unbounded support.
\end{proposition}
\begin{proof}
The assumptions imply that LSD for $\frac{A}{\sqrt{n}}$ and $\frac{B}{\sqrt{n}}$ exists.
By Theorem~\ref{main theorem-1},
$\{\frac{A}{\sqrt{n}},\frac{B}{\sqrt{n}}\}$ converge jointly and hence
$\lim_{n\rightarrow \infty}\frac{1}{n^{k/2+1}}E(\Tr(A+B)^k=\beta_k$ exists for all $k>0$. Now let us fix $k$. Let $Q_k$ be the set of monomials such that $(A+B)^k=\sum_{q\in Q_k}q(A,B).$ Hence
\begin{equation*}
\frac{1}{n}\Tr(\frac{A+B}{\sqrt{n}})^k=\frac{1}{n^{1+k/2}}\sum_{q\in Q_k}\Tr(q(A,B))\\
=\sum_{q\in Q_k}\widehat{\mu_n}(q)
\end{equation*}
where $\widehat{\mu_n}(q)$ is as in Section~\ref{sec:def}. By (3) of  Theorem~\ref{main theorem-1},  $\widehat{\mu_n}(q)\rightarrow \alpha(q)$, almost surely and hence,
$$\beta_k=\lim_{n\rightarrow \infty}\frac{1}{n}\Tr(\frac{A+B}{\sqrt{n}})^k=\sum_{q\in Q_k}\alpha(q)~~~almost~ surely.$$
Using (2) of Theorem~\ref{main theorem-1}, we have
$$\beta_{2k}= \sum_{q\in Q_{2k}}\alpha(q) \leq |Q_{2k}|\frac{(2k)!}{k!2^k}\Delta(L_1,L_2)^{k}=2^{2k}\frac{(2k)!}{k!2^k}\Delta(L_1,L_2)^{k}.$$
Now by using Stirling's formula,
$\beta_{2k}\leq (Ck)^k$ for some constant $C$. Hence
$\sum_k \beta_{2k}^{-1/2k}=\infty$ and \emph{Carleman's
Condition} is satisfied implying that the LSD exists.


To prove symmetry of the limit, let $q\in Q_{2k+1}$. Then from (2) of Theorem~\ref{main theorem-1}, it follows that $\alpha(q)=0.$
Hence  $\beta_{2k+1}=\sum_{q\in Q_{2k+1}}\alpha(q)=0$ and the distribution is symmetric.

To prove unboundedness, without loss of generality let us assume that LSD  $\mathcal{L}_A$ of $\frac{A}{\sqrt{n}}$ has unbounded support. Let us denote by $\beta_{2k}(A)$ the $(2k)$th moment of $\mathcal{L}_A$. Since $L^p$ norm converges to essential supremum as $p \rightarrow \infty$ it follows that $(\beta_{2k}(A))^{1/2k}\rightarrow \infty$ as $k\rightarrow \infty$. Also, $\beta_{2k}(A)= \alpha(q_{2k})$ where $q_{2k}(A,B)=A^{2k}$ and $q_{2k}\in Q_{2k}$. Since $\alpha(q)$ is non-negative for all $q$, it implies $\beta_{2k}\geq \beta_{2k}(A)$. So $\lim_{k\rightarrow \infty}(\beta_{2k})^{1/2k}=\infty $ and hence the LSD of $\frac{A+B}{\sqrt{n}}$ has unbounded support.
\end{proof}
In particular,
all conclusions in Proposition~\ref{prop:sum} hold when
$A$ and $B$  are any two of  Toeplitz, Hankel, Reverse Circulant and Symmetric Circulant matrices.
It does not seem easy to identify the LSD for these sums.
Some simulation results are given below.

When one of the matrix is Wigner,
Theorem~\ref{theorem:free1} implies that the limit is the free convolution of the semicircular law
and the corresponding LSD. This result about the sum when one of them is Wigner also follows from the results of \citet{pastur:vasilchuk}. It also follows from the work of \citet{biane} that any free convolution with the semi-circular law is continuous and the density can be expressed in terms of Stieltjes transform of the LSD. Unfortunately, the Stieltjes transform of the LSD of the Toeplitz and Hankel are
not known.

\vskip-30pt
\begin{figure}[htp]\label{fig:symcircnormalbino}
\centering
\includegraphics[height=60mm, width =60mm ]{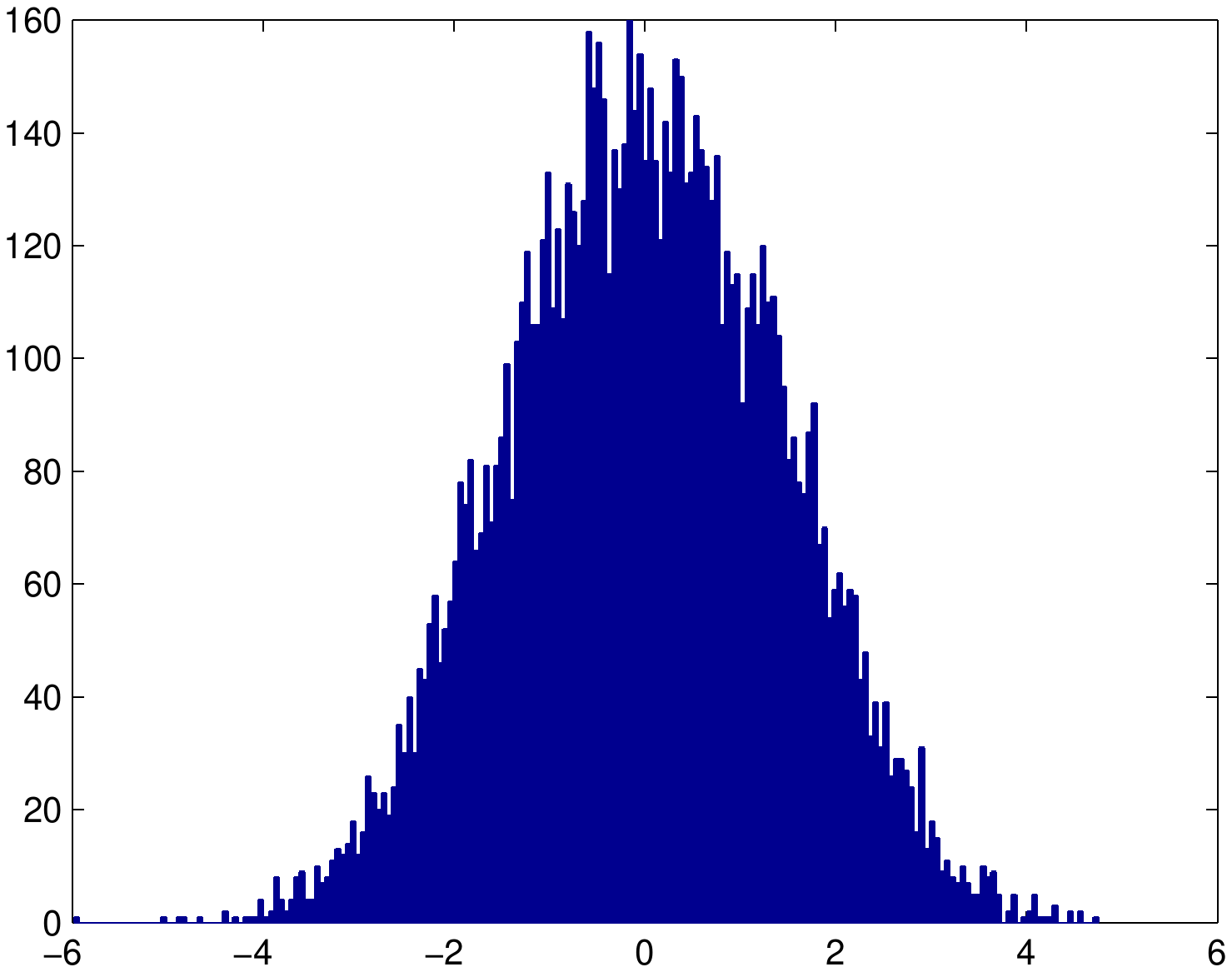}
\includegraphics[height=60mm, width =60mm ]{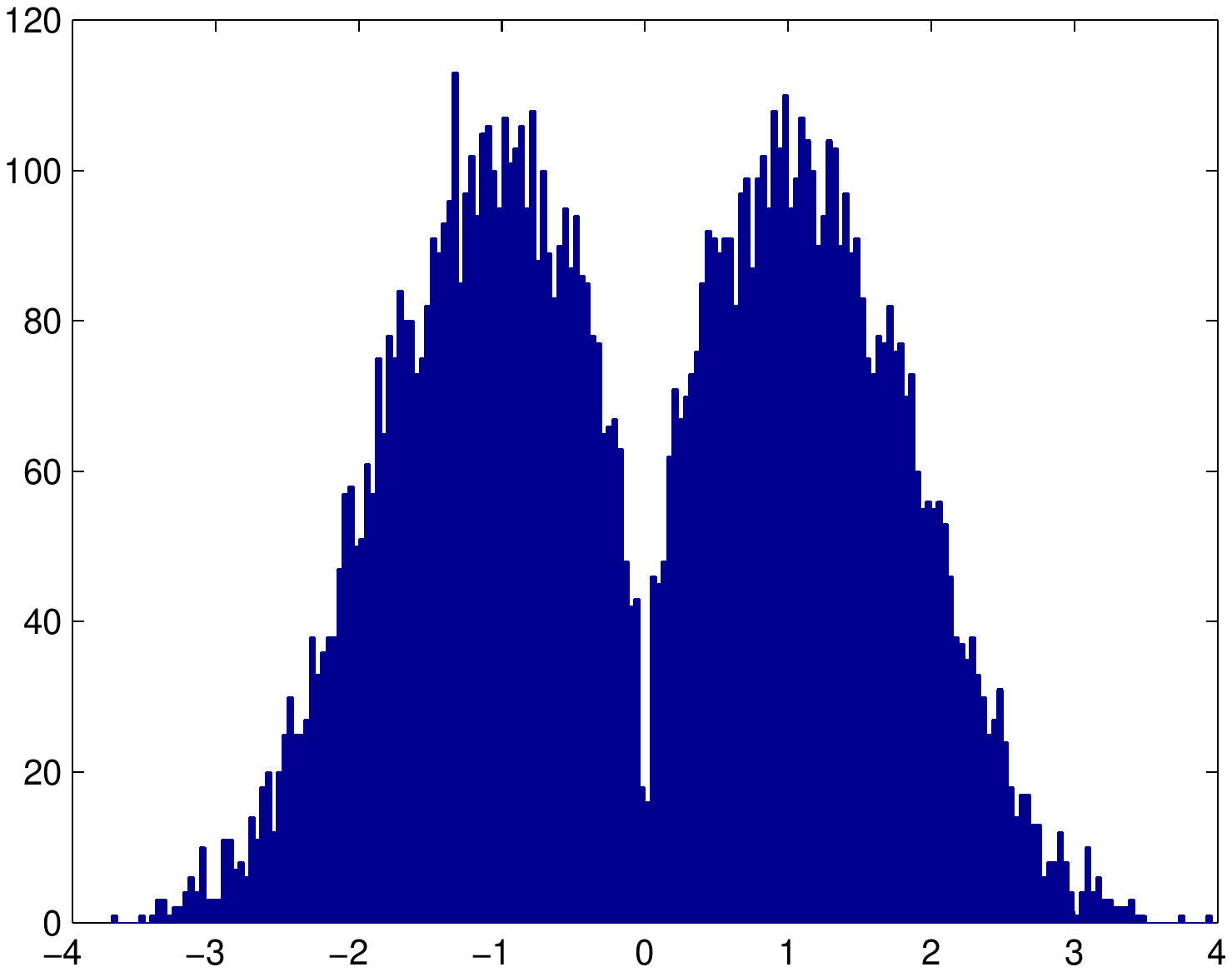}
\vskip-50pt\caption{\footnotesize (i) (left) Histogram plot of empirical distribution of Reverse Circulant+ Symmetric Circulant ($n=500$) with entries $N(0,1)$ (ii) (right) Histogram plot of empirical distribution of Reverse Circulant+Hankel ($n=500$) with $N(0,1)$ entries.}
\label{fig:case1&2}
\end{figure}

\vskip-40pt
\begin{figure}[htp]\label{fig:revcricnormalbino}
\centering
\includegraphics[height=60mm, width =60mm ]{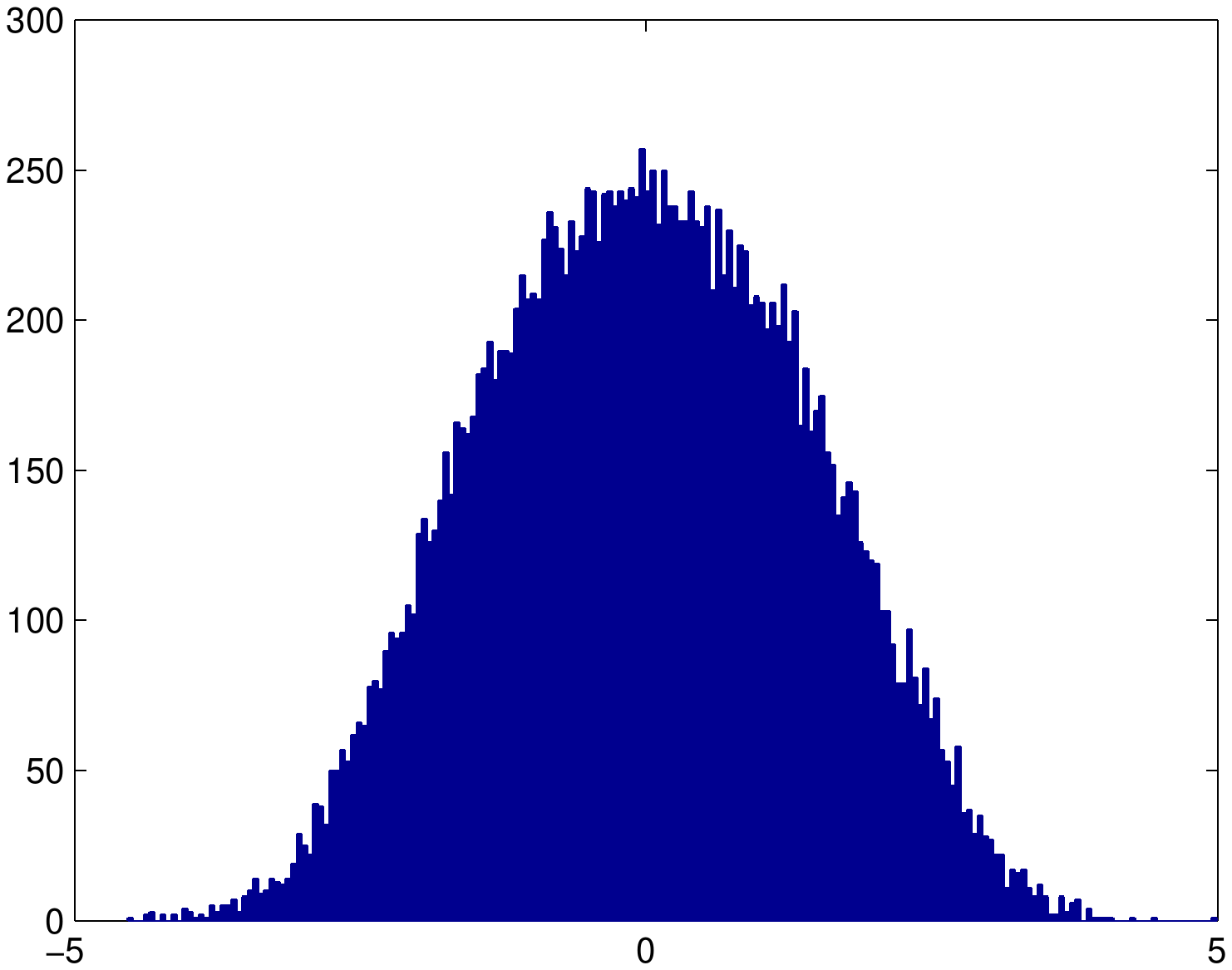}
\includegraphics[height=60mm, width =60mm ]{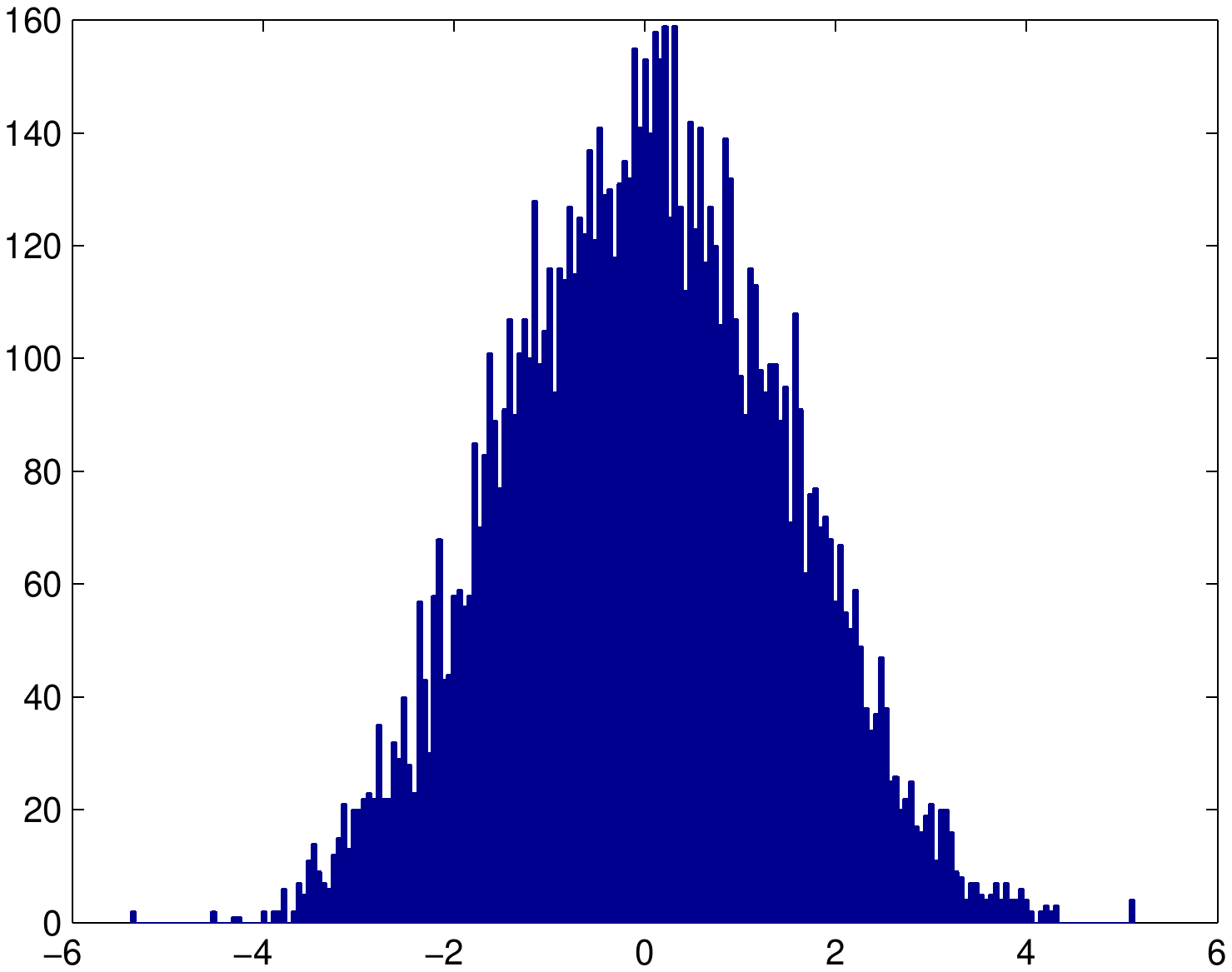}
\vskip-50pt\caption{\footnotesize (i) (left) Histogram plot of empirical distribution of Toeplitz+Hankel($n=1000$) with entries $N(0,1)$ (ii) (right)  Histogram plot of empirical distribution of Toeplitz+Symmetric Circulant ($n=500$) with $N(0,1)$ entries.}\label{fig:case3&4}
\end{figure}
\normalsize

\section{Proofs}\label{sec:proof}
{\it To simplify the notational aspects
in all our proofs we restrict ourselves to $h=2$.}
\subsection{Proof of Theorem~\ref{main theorem-1}}\label{sec:proof:main theorem1}
\noindent (1)  We first show that
\begin{equation}\label{psiw} \Pi_{C}^*(w)=\Pi_{CI}^*(w)\ \ \text{for all} \ \ w\in CIW(2).\end{equation}
Let $\pi \in \Pi_{CI}^*(w)$. As $q$ is fixed,
\begin{eqnarray*}
\psi(w)[i]=\psi(w)[j] &\Rightarrow & w[i]=w[j]\\
\Rightarrow (c_i,t_i, L_{c_i}(\pi(i-1),\pi(i))) &=& (c_j,t_j,
L_{c_j}(\pi(j-1),\pi(j)))~~~~~( \text{as}~ \pi \in \Pi_{CI}^*(w)).
\end{eqnarray*}
This implies $L_{c_i}(\pi(i-1), \pi(i))=L_{c_j}(\pi(j-1), \pi(j))$. Hence $\pi \in \Pi_{C}^*(\psi(w))$.

Now conversely, let $\pi \in \Pi_C^*(\psi (w))$. Then we have
\begin{eqnarray*}
w[i]&=&w[j]\\
\Rightarrow \psi(w)[i]&=&\psi(w)[j]\\
\Rightarrow  L_{c_i}(\pi(i-1), \pi(i))&=&L_{c_j}(\pi(j-1), \pi(j))\\
\Rightarrow  Z_{c_i,t_i}(L_{c_i}(\pi(i-1),\pi(i))) &= &Z_{c_j,t_j}(L_{c_j}(\pi(j-1),\pi(j))).
\end{eqnarray*}
as $w[i]=w[j]\Rightarrow c_i=c_j$ and $t_i=t_j$. Hence $\pi \in \Pi_{CI}^*(w)$.\\
So (\ref{psiw}) is established.
As a consequence,
$$p_{CI}(w)=\lim_{n\rightarrow \infty}\frac{1}{n^{1+k/2}}|\Pi_{CI}^*(w)|= p_C(\psi(w)).$$
Hence by (\ref{psiw}) $p_{CI}(w)$ exists for all $w\in CIW(2)$ and $p_{CI}(w)=p_C(\psi(w))$, proving \eqref{main theo-1:claim 1}.\vskip5pt

\noindent  (2)
Recall that $\mathbf Z_{\pi}=\prod_{j=1}^k Z_{c_j,t_j}(L_{c_ij}(\pi
(j-1), \pi(j))$ and using  (\ref{eq:zeepialmostsure}) and (\ref{eq:nonmatch})
\begin{equation}
\label{matched}
\widehat{\mu}_n(q)=\frac{1}{n^{1+k/2}}\sum_{w:~ w~matched} \sum_{\pi \in \Pi_{CI}(w)}\E(\mathbf Z_\pi).
\end{equation}
By using Assumption ~\eqref{A2}
\begin{equation}
\label{eq:boundedexp}
\sup_{\pi}\E|\mathbf Z_{\pi}| < K < \infty.
\end{equation}
By using often used arguments of \citet{bosesen} and of \citet{bry}, for
any colored and indexed matched word $w$ which is matched but is not
pair-matched,
\begin{equation}
 \label{eq:nonpair}
\lim_{n\rightarrow \infty}\frac{1}{n^{1+k/2}}\big|\sum_{\pi \in \Pi_{CI}(w)}\E(\mathbf Z_\pi)\big| \leq \frac{K}{n^{1+k/2}}\left|\Pi_{CI}(w)\right|\to 0.
\end{equation}
By using (\ref{eq:nonpair}), and  the fact that $\E(\mathbf Z_{\pi})=1$ for every color index pair-matched word (use Assumption~\eqref{A2}), calculating the limit in
(\ref{matched}) reduces to calculating $\lim \frac{1}{n^{1+k/2}}\sum_{w:~ w\in CIW(2)}|\Pi_{CI}(w)|.$

Now consider any $w\in CIW(2)$. Observe that any circuit in $\Pi_{CI}^*(w)-\Pi_{CI}(w)$ must have an edge of order four. Hence  by (\ref{eq:nonpair}),

$$\lim_{n\rightarrow
\infty}\frac{|\Pi_{CI}^*(w)-\Pi_{CI}(w)|}{n^{1+k/2}}=0.$$
As a consequence, since there are finitely many words,
\begin{equation}
\lim_{n \rightarrow \infty}\widehat{\mu}_n(q)= \lim_{n \rightarrow \infty}\sum_{w\in CIW(2)}\frac{|\Pi_{CI}(w)|}{n^{1+k/2}}=\lim_{n \rightarrow \infty}\sum_{w\in CIW(2)}\frac{|\Pi_{CI}^*(w)|}{n^{1+k/2}}=\sum_{w\in CIW(2)}p_{CI}(w)=\alpha(q).
\end{equation}
\comment{
From these above observations and Assumption~\eqref{A1} it follows that
\begin{equation}\label{eq:three limits}
\lim_{n\rightarrow \infty}\widehat{\mu}_n(q)=\lim_{n\rightarrow \infty}\frac{1}{n^{1+k/2}}\sum_{w:~ w\in CIW(2)}|\Pi_{CI}^*(w)|=\lim_{n\rightarrow \infty}\frac{1}{n^{1+k/2}}\sum_{w:~ w\in CIW(2)}|\Pi_{CI}(w)|
\end{equation}
in the sense that if either of the three limits exist in \eqref{eq:three limits} then the other two exist as well and they are equal.
 $w$ be a (colored and indexed) matched word which is not
pair-matched. Then
\begin{equation}
 \label{non pair matched vanishing}
\lim_{n\rightarrow \infty}\frac{1}{n^{1+k/2}}\sum_{\pi \in \Pi_{CI}(w)}\E(\mathbf Z_\pi)=0.
\end{equation}
 As there are only finitely many words corresponding to $q$, we have from \eqref{matched} and~\eqref{non pair matched vanishing},
\begin{equation}
\lim_{n \rightarrow \infty}\widehat{\mu}_n(q)=\lim_{n \rightarrow \infty}\frac{1}{n^{1+k/2}}\sum_{w:~ w\in CIW(2)} \sum_{\pi \in \Pi_{CI}(w)}\E(\mathbf Z_{\pi}).
\end{equation}
By Assumption~\eqref{A1}, $\E(Z_{\pi})=1$ for all $\pi \in \Pi_{CI}(w)$ if $w$ is pair-matched. Hence,
}
To complete the proof of (2), we note that, if either $k$ is odd or some index appears an odd number of times in $q$
then for that $q$, $CIW(2)$ is empty and hence, $\alpha(q)=0$. If $k$ is even and every index appears an even number of times, then we know
$$|CIW(2)|\leq |CW(2)|\leq \frac{k!}{{(k/2)}!2^{k/2}}.$$ Now note that $p_{CI}(w)\leq \Delta^{k/2}$. Combining all these, we get $|\alpha(q)|\leq \frac{k!\Delta^{k/2}}{(k/2)!2^{k/2}}$.\vskip5pt

\noindent  (3) Now we claim that  $$\E[(\widetilde{\mu
_{n}}(q)-\widehat{\mu}_n(q))^4]= O(n^{-2}).$$ Observe that,
\begin{equation}
\label{as1}
\E[(\widetilde{\mu _{n}}(q)-\widehat{\mu}_n(q))^4]=\frac{1}{n^{2k+4}}\sum_{\pi_1,\pi_2,\pi_3,\pi_4}\E(\prod_{j=1}^4(\mathbf Z_{\pi_j}-E(\mathbf Z_{\pi_j})).
\end{equation}
We say $(\pi_1,\pi_2,\pi_3,\pi_4)$ are ``jointly matched" if each
$L$-value occurs at least twice across all circuits (among same
color) and they are said to be ``cross matched" if each circuit has
at least one $L^*$ value which occurs in some other circuit.

If $(\pi_1,\pi_2,\pi_3,\pi_4)$ are not jointly matched then without
loss of generality  there exist some $L$-value in $\pi_1$ which does
not occur anywhere else. Using $\E(\mathbf Z_{\pi_1})=0$ and
independence,
\begin{equation}
\E(\prod_{j=1}^4(\mathbf Z_{\pi_j}-\E(\mathbf Z_{\pi_j}))=\E(\mathbf Z_{\pi_1}\prod_{j=2}^4(\mathbf Z_{\pi_j}-\E(\mathbf Z_{\pi_j}))= 0.
\end{equation}
Again, if $(\pi_1,\pi_2,\pi_3,\pi_4)$ are jointly matched but not cross matched, then without loss of generality, assume $\pi_1$ is only self matched. Then by independence,
\begin{equation}
\label{as2}
\E(\prod_{j=1}^4(\mathbf Z_{\pi_j}-\E(\mathbf Z_{\pi_j}))=\E[\mathbf Z_{\pi_1}-\E(\mathbf Z_{\pi_1})]\E[\prod_{j=2}^4(\mathbf Z_{\pi_j}-\E(\mathbf Z_{\pi_j}))]= 0.
\end{equation}
So we are left with circuits that are jointly matched and cross
matched with respect to $q$. Let $Q_q$ be the number of such
circuits.

We claim that $Q_q=O(n^{2k+2})$.
Since the circuits are jointly matched there are at most $2k$ distinct $L$ values among all the four circuits. Let $u$ be the number of distinct $L$ values (of a single color)  in the circuits. Clearly, for a fixed choice of matches among those  distinct $L$ values (number of such choices is bounded in $n$), the number of jointly matched and cross matched circuits are $O(n^{u+4})$, so the number of such circuits with $u\leq 2k-2$ is $O(n^{2k+2})$. Hence it suffices to prove that for a fixed choice of matches among $u=2k-1$ or $u=2k$ distinct $L$-values occurring across all four circuits, the number of  jointly matched and cross matched circuits is $O(n^{2k+2}).$

We
 consider only the case $u=2k-1$ and the
other case is dealt in a similar way.
 Since $u=2k-1$, it follows that every $L$-value occurs exactly twice across all four circuits. Since $\pi_1$ is not self matched, there is
 an $L$ value in $\pi_1$ which does not occur anywhere else in $\pi_1$.  We consider the following dynamic construction of $(\pi_1,\pi_2,\pi_3,\pi_4)$. Since the circuit is cross matched, there exists an $L$ value which is assigned to a single edge, say $L(\pi_1(i_*-1),\pi(i_*))$. First choose one of the $n$ possible values for the initial value $\pi_1(0)$, and continue
filling in the values of $\pi_1(i), i = 1, 2, . . . , i_*-1.$ Then, starting at $\pi_1(k) = \pi_1(0)$, sequentially choose the values of $\pi_1(k-1),\pi_1(k-2), . . . , \pi_1(i_*),$ thus completing the entire circuit $\pi_1$. At every stage there are $n$ ways to choose a vertex if there is no $L$-match of the edge being constructed with the previously constructed edges, otherwise there are at most $\Delta(L_1,L_2)$ choices.
 So there are $O(n)$ choices for at most $2k-2$ distinct $L$ values and hence the number of jointly matched and cross matched circuits for $u=2k-1$ is $O(n^{2k-2+4})$, as required.

By Assumption~\eqref{A2}, $\E[\prod_{j=1}^4(\mathbf Z_{\pi_j}-\E(\mathbf Z_{\pi_j}))]$ is uniformly bounded over all $(\pi_1,\pi_2,\pi_3,\pi_4)$ by $K$, say. By this and (\ref{as1})--(\ref{as2}), it follows that
\begin{equation}
\E[(\widetilde{\mu _{n}}(q)-\widehat{\mu}_n(q))^4]=O(\frac{n^{2k+2}}{n^{2k+4}})=O(n^{-2}).
\end{equation}
Now using
Borel-Cantelli Lemma,  $\widetilde{\mu _{n}}(q)-\widehat{\mu}_n(q)\rightarrow 0$ almost surely as $n\rightarrow \infty$ and this completes the proof.
\subsection{Proof of Theorem~\ref{main theorem-2}}
Condition~\eqref{psi} which needs to be verified (only for even degree monomials), crucially depends
on the type of the link function and hence we need to deal with
every example differently. Since we are dealing with only two link
functions, we
simplify the notation. Let $X$ and $Y$ be
patterned matrices with link function $L_1$ and $L_2$ respectively
with independent input sequences satisfying Assumptions~\eqref{A1}
and
\eqref{A2}. Let $q(X,Y)$ be any monomial such that
both $X$ and $Y$ occur an even number of times in $q$. Let
$deg(q)=2k$ and let the number of times $X$ and $Y$ occurs in the monomial
be $k_1$ and $k_2$ respectively. Note that we have $k=k_1+k_2$. Then
it is enough to show that~\eqref{psi} holds for every pair-matched colored word $w$ of
length $2k$ corresponding to $q$.

Let $X$ and $Y$ be any of the two following matrices: Wigner($W_n$), Toeplitz($T_n$), Hankel($H_n$), Reverse Circulant($RC_n$) and Symmetric Circulant($SC_n$).
The case when both $X$ and $Y$ are of the same pattern was dealt in \citet{bosehazrasaha}.

Proof of Theorem~\ref{main theorem-2} is immediate once we establish the following Lemma.
\begin{lemma}
\label{lemma:theorem-2} Let $X$ and $Y$ be any of the
matrices, $W_n, T_n, H_n, RC_n$ and $SC_n$,  satisfying
Assumption~\eqref{A2}. Let $w\in CW(2)$  corresponding to a monomial
$q$ of length $2k$. Then there exists a (finite) index set $I$
independent of $n$ and $\{\Pi_{C,l}^*(w):l\in I\}\subset
\Pi_{C}^*(w)$  such that \vskip5pt


\noindent  (1) $\quad \Pi_{C}^*(w)= \cup_{l\in I} \Pi_{C,l}^*(w), \text{ and } p_{C,l}(w):=
\lim_{n\rightarrow\infty}\frac{|\Pi_{C,l}^*(w)|}{n^{1+k}} \quad\text{exists for all} \ \ l\in I,$\vskip5pt

\noindent  (2)$\quad \text{for} \ \ l\neq l^\prime\ \  \text{we
have}$,
\begin{equation}\label{eq:intersection} | \Pi_{C,l}^*(w)\cap
\Pi_{C,l^\prime}^*(w)|=\lito(n^{1+k}).\end{equation}
\end{lemma}
Assuming  Lemma~\ref{lemma:theorem-2},
 $|\Pi_{C}^*(w)|= |\cup_{l\in I} \Pi_{C,l}^*(w)|$ for some finite index set $I$ and
\begin{equation}
p_C(w)= \lim_{n\rightarrow \infty}\frac{1}{n^{1+k}}|\Pi_{C}^*(w)|
= \sum_{l\in I} \lim_{n\rightarrow \infty}\frac{1}{n^{1+k}}|\Pi_{C,l}^*(w)|= \sum_{l\in I} p_{C,l}(w).
\end{equation}
The proof of this lemma treats each pair of matrices separately. Since the arguments are similar for the different pairs, we do not provide the detailed proof for each case but only a selection of the arguments in most cases.

The set $S$ of all generating vertices  of $w$ is split into the three classes $\{0\}\cup S_X \cup S_Y$ where
\begin{equation*}
 S_X=\{i\wedge j : c_i=c_j=X, \ \ w[i]=w[j] \},\ \
S_Y =\{i\wedge j : c_i=c_j=Y, \ \ w[i]=w[j] \}.
\end{equation*}
For every $i \in S-\{0\}$, let $j_i$ denote the index such that $w[j_i]=w[i]$.  Let $\pi \in \Pi_C^*(w)$.
\vskip5pt
\noindent
(i) \textit{Toeplitz and Hankel:} Let $X$ and $Y$ be respectively the Toeplitz ($T$)  and the Hankel ($H$) matrix.
Observe that,
$$|\pi(i-1)-\pi(i)|=|\pi(j_i-1)-\pi(j_i)| \text{ for all } i \in S_T$$
$$\pi(i-1)+\pi(i)=\pi(j_i-1)+\pi(j_i) \text{ for all } i \in S_H.$$
Let $I$ be $\{-1,1\}^{k_1}$ and $l=(l_1,...,l_{k_1}) \in I$. Let $\Pi_{C,l}^*(w)$ be the subset of $\Pi_C^*(w)$ such that,
$$\pi(i-1)-\pi(i)=l_i(\pi(j_i-1)-\pi(j_i))\qquad\text{ for all } i \in S_T,$$
$$\pi(i-1)+\pi(i)=\pi(j_i-1)+\pi(j_i)\qquad \text{ for all } i \in S_H. $$
Now clearly, $$\Pi_C^*(w)=\bigcup_{l} \Pi_{C,l}^*(w)\text{ (not a disjoint union).}$$
Now let us define, \begin{equation}\label{eq:vertex} v_i=\frac{\pi(i)}{n} \qquad\text{and}\qquad U_n=\{0,\frac{1}{n},...,\frac{n-1}{n}\}.\end{equation}
Then,
\begin{eqnarray*}|\Pi_{C,l}^*(w)|&=&\# \{(v_0,...,v_{2k}):v_i\in U_n ~\forall 0\leq i\leq 2k,~~ v_{i-1}-v_i=l_i(v_{j_i-1}-v_{j_i})~~\forall i \in S_T \\
 &&\ \ \ \ \ \ \text{ and }v_{i-1}+v_i=v_{j_i-1}+v_{j_i}~~\forall i \in S_H,~~v_0=v_{2k}\}.\end{eqnarray*}
Let us denote $\{v_i:i\in S\}$ by $v_S$. It can easily be seen from the above equations (other than $v_0=v_{2k}$) that each of the $\{v_i: i\notin S\}$ can be written uniquely as an integer linear combination $L_i^l(v_S)$. Moreover, $L_i^l(v_S)$ only contains $\{v_j:\ j\in S,~j<i\}$ with non-zero coefficients.
Clearly,
\begin{equation}
|\Pi_{C,l}^*(w)|=\# \{(v_0,...,v_{2k}): v_i\in U_n ~\forall 0\leq i\leq 2k,v_0=v_{2k}, v_i=L_i^l(v_S)\ \  \forall i\notin S\}.
\end{equation}
 Any integer linear combinations of elements of $U_n$ is again in $U_n$ if and only if it is between 0 and 1. Hence,
\begin{equation}
\label{riemann}
|\Pi_{C,l}^*(w)|=\# \{v_S: v_i\in U_n ~\forall i\in S, v_0=L_{2k}^l(v_S), 0\leq L_i^l(v_S)< 1 \ \ \forall i\notin S\}.
\end{equation}
From (\ref{riemann}) it follows that, $\frac{|\Pi_{C,l}^*(w)|}{n^{1+k}}$ is nothing but the Riemann sum for the function\\ $I(0\leq L_i^l(v_S)< 1, i\notin S, v_0=L_{2k}^l(v_S))$ over $[0,1]^{k+1}$ and converges to the integral and hence
\begin{equation}
p_{C,l}(w)= \lim_{n\rightarrow \infty}\frac{1}{n^{1+k}}|\Pi_{C,l}^*(w)|=\int\limits_{[0,1]^{k+1}}
I\left(0\leq L_i^l(v_S)< 1, i\notin S, v_0=L_{2k}^l(v_S)\right)dv_S.
\end{equation}
This shows part (1) of Lemma~\ref{lemma:theorem-2}. For  part (2) let $l\neq l'$. Without loss of generality, let us assume that, $l_{i_1}=-l_{i_1}'$. Let $\pi \in \Pi_{C,l}^*(w)\bigcap \Pi_{C,l'}^*(w)$.  Then $\pi (i_1-1)=\pi(i_1)$ and hence $L_{{i_1}-1}^l(v_S)=v_{i_1}$. It now follows along the lines of the preceding arguments that
\begin{equation}
\lim_{n\rightarrow \infty}\frac{1}{n^{1+k}}|\Pi_{C,l}^*(w)\bigcap \Pi_{C,l'}^*(w)| \leq \idotsint\limits_{[0,1]^{k+1}}
I(v_i=L_{{i_1}-1}^l(v_S))dv_S.
\end{equation}
$L_{{i_1}-1}^l(v_S)$ contains $\{v_j:\ j\in S,~j<i_1\}$ and hence $\{L_{{i_1}-1}^l(v_S)=v_i\}$ is a $k$-dimensional subspace of $[0,1]^{k+1}$ and hence has Lebesgue measure 0.\vskip5pt


\noindent
(ii) \textit{Hankel and Reverse Circulant:} Let $X$ and $Y$ be Hankel ($H$) and Reverse Circulant ($RC$) respectively.
Then
\begin{equation}\label{eq:hankeleq}\pi(i-1)+\pi(i)=\pi(j_i-1)+\pi(j_i)\quad \text{for all } i \in S_H,
\end{equation}
\begin{equation}\label{eq:reverseeq}
(\pi(i-1)+\pi(i))\mod~n=(\pi(j_i-1)+\pi(j_i))\mod~n\quad \text{for all } i \in S_{RC}.
\end{equation}
Clearly, as all the $\pi(i)$ are between 1 and $n$, relation (\ref{eq:reverseeq})
implies $(\pi(i-1)+\pi(i))-(\pi(j_i-1)+\pi(j_i))=a_in$ where $a_i\in \{0,1,-1\}$

Let $a=(a_{1},...,a_{k_2}) \in I=\{-1,0,1\}^{k_2}$. Let
$\Pi_{C,a}^*(w)$ be the subset of $\Pi_C^*(w)$ such that,
$$\pi(i-1)+\pi(i)=\pi(j_i-1)+\pi(j_i)~~\forall i \in S_H \text{ and }$$
$$(\pi(i-1)+\pi(i))-(\pi(j_i-1)+\pi(j_i))=a_in~~~\forall i \in S_{RC}.$$
Now clearly, $$\Pi_C^*(w)=\bigcup_{a} \Pi_{C,a}^*(w) \text{ (a disjoint union).}$$
Now we get that,
\begin{eqnarray*}
|\Pi_{C,a}^*(w)|&=&\# \{(v_0,...,v_{2k}): v_i\in U_n ~\forall 0\leq i\leq 2k,~ v_{i-1}+v_i=(v_{j_i-1}+v_{j_i})+a_i~~\forall i \in S_{RC}\\
&&\ \ \ \ \ \ \ \text{and}\ \ v_{i-1}+v_i=v_{j_i-1}+v_{j_i}~~\forall i \in S_H,~~v_0=v_{2k}\}.
\end{eqnarray*}
Other than $v_0=v_{2k}$,   each  $\{v_i: i\notin S\}$ can be written uniquely as an affine linear combination $L_i^a(v_S)+b_i^{(a)}$ for some integer $b_i^{(a)}$. Moreover, $L_i^a(v_S)$ only contains $\{v_j:\ j\in S,~j<i\}$ with non-zero coefficients.
Arguing as in the previous case,
\begin{equation}
\label{riemann2}
|\Pi_{C,a}^*(w)|=\# \{v_S: v_i\in U_n ~\forall i\in S, v_0=L_{2k}^a(v_S)+b_{2k}^{(a)}, 0\leq L_i^a(v_S)+ b_i^{(a)}< 1 \forall i\notin S\}.
\end{equation}
This is again a Riemann sum and hence as before,
\begin{equation*}
    p_{C,a}(w) =\lim_{n\rightarrow \infty}\frac{1}{n^{1+k}}|\Pi_{C,a}^*(w)|
=\int\limits_{[0,1]^{k+1}}
I\left(0\leq L_i^a(v_S)+b_i^{(a)} <1, i\notin S, v_0=L_{2k}^a(v_S)+b_{2k}^{(a)}\right)dv_S
\end{equation*}
and the proof of this case is complete. \vskip5pt

\noindent
(iii)  \textit{Hankel and Symmetric Circulant:} Let $X$ and $Y$ be Hankel ($H$) and Symmetric Circulant~($SC$) respectively.  Note that
$$\pi(i-1)+\pi(i)=\pi(j_i-1)+\pi(j_i)~~\forall i \in S_H ~\text{and }$$
$$n/2-|n/2-|\pi(i-1)-\pi(i)||=n/2-|n/2-|\pi(j_i-1)-\pi(j_i)||~~\forall i \in S_S. $$
It can be easily seen from the second equation above that either $|\pi(i-1)-\pi(i)|= |\pi(j_i-1)-\pi(j_i)|$ or $|\pi(i-1)-\pi(i)|+|\pi(j_i-1)-\pi(j_i)|=n$. There are six cases for each Symmetric Circulant match $[i,j_i]$, and with $v_i=\pi(i)/n$, these are:
\begin{enumerate}
\item{$v_{i-1}-v_i- v_{j_i-1}+v_{j_i}=0$.}
\item{$v_{i-1}-v_i+ v_{j_i-1}-v_{j_i}=0$.}
\item{$v_{i-1}-v_i+ v_{j_i-1}-v_{j_i}=1$.}
\item{$v_{i-1}-v_i- v_{j_i-1}+v_{j_i}=1$.}
\item{$v_{i}-v_{i-1}+ v_{j_i-1}-v_{j_i}=1$.}
\item{$v_{i}-v_{i-1}+ v_{j_i}-v_{j_i-1}=1$.}
\end{enumerate}
Now we can write $\Pi_C^*(w)$ as the (not disjoint) union  of $6^{k_2}$ possible $\Pi_{C,l}^*(w)$ where $l$ denotes the combination of cases (1)--(6) above that is satisfied in the $k_2$ matches of Symmetric Circulant. For each  $\pi \in \Pi_{C,l}^*(w)$, each $\{v_i: i\notin S\}$ can be written uniquely as an affine integer combination of $v_S$. As in the previous two pairs of matrices in (i) and (ii),
$\lim_{n\rightarrow \infty}\frac{1}{n^{1+k}}|\Pi_{C,l}^*(w)|$ exists as an integral.
\comment
{Now, it remains to show that
\begin{equation}
\label{disjoint-1}
\lim_{n\rightarrow \infty}\frac{1}{n^{1+k}}|\Pi_{C,l}^*(w)\bigcap \Pi_{C,l'}^*(w)|=0\ \ \text{for}\ \ l \neq l'.
\end{equation}
}

Now (\ref{eq:intersection})  can be checked case by case.
As a typical case suppose Case 1 and Case 3 hold.
Then $\pi(i-1)-\pi(i)=n/2$ and $v_{i-1}-v_i=1/2$. Since $i$ is generating and $v_{i-1}$ is a linear combination of $\{v_j:\ j\in S,~j<i\}$, this implies a non-trivial linear relation between the independent vertices  $v_S$. This, in turn implies that the number of circuits $\pi$ satisfying the above conditions is $o(n^{1+k})$.
\vskip5pt

\noindent
(iv) \textit{Toeplitz and Symmetric Circulant:}
Let $X$ and $Y$ be Toeplitz ($T$) and Symmetric Circulant ($SC$) respectively. Again note that,
$$|\pi(i-1)-\pi(i)|=|\pi(j_i-1)-\pi(j_i)|~~\forall i \in S_T ~\text{and }$$
\begin{equation}\label{eq:symcirmatch}
n/2-|n/2-|\pi(i-1)-\pi(i)||=n/2-|n/2-|\pi(j_i-1)-\pi(j_i)||~~\forall i \in S_{SC}.
\end{equation}
Now,
(\ref{eq:symcirmatch}) implies
either $|\pi(i-1)-\pi(i)|= |\pi(j_i-1)-\pi(j_i)|$ or $|\pi(i-1)-\pi(i)|+|\pi(j_i-1)-\pi(j_i)|=n$.

There are six cases for each Symmetric Circulant match as in Case (iii) above and two cases for each  Toeplitz match.

As before we can write $\Pi_C^*(w)$ as the (not disjoint) union of $2^{k_1}\times 6^{k_2}$ possible $\Pi_{C,l}^*(w)$ where $l$ denotes a combination of cases (1)--(6)  for all $SC$ matches (as in Case (iii)) and a combination of cases (1)-(2) for all $T$ matches. As before, for each  $\pi \in \Pi_{C,l}^*(w)$, each of the $\{v_i: i\notin S\}$ can be written uniquely as an affine integer combination of $v_S$.
As
earlier, $\lim_{n\rightarrow \infty}\frac{1}{n^{1+k}}|\Pi_{C,l}^*(w)|$ exists as an integral.

Now, (\ref{eq:intersection})
\comment{
\begin{equation}
\label{disjoint-2}
\lim_{n\rightarrow \infty}\frac{1}{n^{1+k}}|\Pi_{C,l}^*(w)\bigcap \Pi_{C,l'}^*(w)|=0 \ \ \text{for }l \neq l'.
\end{equation}
}
is
again checked case by case.
Suppose $l\neq l'$ and $\pi \in \Pi_{C,l}^*(w)\bigcap \Pi_{C,l'}^*(w)$. For  $l\neq l'$, there must be one Toeplitz or Symmetric Circulant match such that two of the possible cases in (1)--(2) or in (1)--(6) occur simultaneously.  Here we just deal with a typical pair
Case (1) and Case (2) for the Toeplitz match. Then we have $\pi(i-1)-\pi(i)=0$ and hence $v_{i-1}-v_i=0$. Since $i$ is generating and $v_{i-1}$ is a linear combination of $\{v_j:\ j\in S,~j<i\}$, this implies there exist a non-trivial relation between the independent vertices  $v_S$. This, in turn implies that the number of circuits $\pi$ satisfying the above conditions in $o(n^{1+k})$.
Now suppose the Symmetric Circulant match happens for both case (1)  and case (2).
Then again we have $v_i=v_{i-1}$ and we can argue as before to conclude that (\ref{eq:intersection}) holds.
\vskip5pt

\noindent
(v)  \textit{Toeplitz and Reverse Circulant:}
Let $X$ and $Y$ be Toeplitz ($T$) and Reverse Circulant ($RC$) respectively.
Note,
$$|\pi(i-1)-\pi(i)|=|\pi(j_i-1)-\pi(j_i)|\quad \text{for all } i \in S_T, $$
$$(\pi(i-1)+\pi(i))\mod~n=(\pi(j_i-1)+\pi(j_i))\mod~n\quad \text{for all } i \in S_{RC}. $$
Clearly, as all the $\pi(i)$ are between 1 and $n$, $(\pi(i-1)+\pi(i))\mod~n=(\pi(j_i-1)+\pi(j_i))\mod~n$ implies $(\pi(i-1)+\pi(i))-(\pi(j_i-1)+\pi(j_i))=a_in$ where $a_i\in \{0,1,-1\}$

Let the number of Toeplitz and Reverse Circulant matches  be $k_1$, and $k_2$ respectively and let us denote $S_T=\{i_1,i_2,...,i_{k_1}\}, S_{RC}=\{i_{k_1+ 1},i_{k_1+2},...,i_{k_1+k_2}\}$.\\ Let $l=(c,a)=(c_{i_1},...,c_{i_{k_1}},a_{i_{k_1+1}},...,a_{i_{k_1+k_3}}) \in I=\{-1,1\}^{k_1}\times \{-1,0,1\}^{k_3}$.

Let $\Pi_{C,l}^*(w)$ be the subset of $\Pi_C^*(w)$ such that,
$$\pi(i-1)-\pi(i)=c_i(\pi(j_i-1)-\pi(j_i))~~\forall i \in S_T$$
$$\pi(i-1)+\pi(i)=\pi(j_i-1)+\pi(j_i)+a_in~~\forall i \in S_{RC}. $$
Now clearly, $$\Pi_C^*(w)=\bigcup_{l\in I} \Pi_{C,l}^*(w),$$ and translating this in the language of $v_i$'s, we get
\begin{eqnarray*}
|\Pi_{C,l}^*(w)|&=&\# \{(v_0,...,v_{2k}): v_i\in U_n ~\forall 0\leq i\leq 2k,~ v_{i-1}+v_i=(v_{j_i-1}+v_{j_i})+a_i~~\forall i \in S_{RC}\\
&& \ \ \ \ \ \ \text{and}\ \ \ \ v_{i-1}-v_i=c_i(v_{j_i-1}-v_{j_i})~~\forall i \in S_T,~~v_0=v_{2k}\}.
\end{eqnarray*}
As in the previous cases,  $\lim_{n\to\infty}\frac{|\Pi_{C,l}^*(w)|}{n^{1+k}}$ exists.
It remains to show that,
$\lim_{n \rightarrow \infty}\frac{|\Pi_{C,l}^*(w)\bigcap \Pi_{C,l'}^*(w)|}{n^{1+k}}=0 \ \ \text{for}\  \ l \neq l'.$
If $l=(c,a)\neq l'=(c',a')$, then either $c \neq c'$ or $a\neq a'$. If $c=c'$, then clearly $\Pi_{C,l}^*(w)$ and $\Pi_{C,l'}^*(w)$ are disjoint. Let $c \neq c'$. Without loss of generality, we assume $c_{i_1}=-c_{i_1}$. Then clearly, for every $\pi \in  \Pi_{C,l}^*(w)\bigcap \Pi_{C,l'}^*(w)$ we have $v_{i_1-1}=v_i$, which gives a non-trivial relation between $\{v_j:\ j\in S\}$. That in turn implies the required limit is 0.\vskip5pt

\noindent
(vi)  \textit{Reverse Circulant and Symmetric Circulant:}
Let $X$ and $Y$ be Reverse Circulant ($RC$) and Symmetric Circulant ($SC$) respectively. Then
$$\pi(i-1)+\pi(i) \mod~n=\pi(j_i-1)+\pi(j_i) \mod~n~~\forall i \in S_{RC} ~\text{and }$$
$$n/2-|n/2-|\pi(i-1)-\pi(i)||=n/2-|n/2-|\pi(j_i-1)-\pi(j_i)||~~\forall i \in S_{SC}.$$
As before, the latter equation implies
either $|\pi(i-1)-\pi(i)|= |\pi(j_i-1)-\pi(j_i)|$ or $|\pi(i-1)-\pi(i)|+|\pi(j_i-1)-\pi(j_i)|=n$.

There are now three cases for each  Reverse Circulant match:
\begin{enumerate}
\item{$v_{i-1}+v_i- v_{j_i-1}-v_{j_i}=0$.}
\item{$v_{i-1}+v_i- v_{j_i-1}-v_{j_i}=1$.}
\item{$v_{i-1}+v_i- v_{j_i-1}-v_{j_i}=-1$.}
\end{enumerate}
Also, there are six cases for each Symmetric Circulant match as in Case (iii).

 As before we can write $\Pi_C^*(w)$ as the
union of $3^{k_1}\times 6^{k_2}$ possible $\Pi_{C,l}^*(w)$. Hence arguing in a similar manner, $\lim_{n\rightarrow
\infty}\frac{1}{n^{1+k}}|\Pi_{C,l}^*(w)|$ exists as an integral.
Now, to check (\ref{eq:intersection}),
\comment{\begin{equation}
\label{disjoint}
\lim_{n\rightarrow \infty}\frac{1}{n^{1+k}}|\Pi_{C,l}^*(w)\bigcap \Pi_{C,l'}^*(w)|=0 \ \ \text{for}  l\neq l'.
\end{equation}
}
case by case. Suppose $l\neq l'$ and $\pi \in \Pi_{C,l}^*(w)\bigcap \Pi_{C,l'}^*(w)$. Since $l\neq l'$, there must be one Reverse Circulant or Symmetric Circulant match such that two of the possible cases (1)--(3) or (1)--(6) (which appear in Case (iii)) occur simultaneously. It is easily seen that such an occurrence is impossible for a Reverse Circulant match. We deal with one typical Symmetric Circulant match. Suppose then
we have both case (1) and case (2). Then again we have $v_i=v_{i-1}$
and as a consequence (\ref{eq:intersection}) holds.
\vskip5pt

\noindent
(vii) \textit{Wigner and Hankel:} Let $X$ and $Y$ be Wigner ($W$) and Hankel ($H$) respectively. Observe that,
\begin{equation}\label{eq:c1:c2}(\pi(i-1),\pi(i))=\begin{cases}
                     (\pi(j_i-1),\pi(j_i))&\quad\text{(Constraint $C1$)}\\
                    (\pi (j_i),\pi(j_i-1)) &\quad\text{(Constraint $C2$, for all $i\in S_W$).}
                    \end{cases}\end{equation}
Also, $\pi(i-1)+\pi(i)=\pi(j_i-1)+\pi(j_i)$ for all $i\in S_H$. So for each Wigner match there are two constraints and hence there are $2^{k_1}$ choices. Let $\lambda$ be a typical choice of $k_1$ constraints and $\Pi^*_{C,\lambda}(w)$ be the subset of $\Pi_C^*(w)$ where the above relations hold. Hence $$\Pi_C^*(w)=\bigcup_\lambda \Pi_{C,\lambda}^*(w) \quad\text{(not a disjoint union).}$$ Now using equation~\eqref{eq:vertex} we have,
\begin{eqnarray*}
|\Pi^*_{C,\lambda}(w)|&=&\#\{(v_{0},v_{1}\ldots v_{2k})\, :0\, \le\, v_{i}\,\le 1,v_{0}=v_{2k},v_{i-1}+v_{i}=v_{j_i-1}+v_{j_i}, ~i\in S_H\\
&& \ \ \ \  v_{i-1}=v_{j_i-1}, v_{i}=v_{j_i}, (C1), v_{i-1}=v_{j_i}\,,v_{i}=v_{j_i-1} (C2), i\in S_W\}.
\end{eqnarray*}
 It can be seen from the above equations that each $v_j$, $j\notin S$ can be written (not uniquely) as a linear combination $L_j^\lambda$ of elements in $v_S$. Hence as before,
\begin{eqnarray*}
|\Pi^*_{C,\lambda}(w)|&=&\#\{ v_S: v_i=L_i^\lambda(v_S), v_0=v_{2k}, \text{ for }i\notin S, ,v_{i-1}+v_{i}=v_{j_i-1}+v_{j_i}, ~i\in S_H\\
 &&\ \ \ \  v_{i-1}=v_{j_i-1}, v_{i}=v_{j_i}, (C1), v_{i-1}=v_{j_i}\,,v_{i}=v_{j_i-1} (C2),i\in S_W\}.\end{eqnarray*} So the limit of $|\Pi^*_{C,\lambda}(w)|/n^{1+k}$ exists and can be expressed as an appropriate Riemann sum.

Now we show (\ref{eq:intersection}).
Without loss of generality assume $\lambda_1$ is a $C_1$ constraint and $\lambda_2$ is a $C_2$ constraint. For any $\pi\in \Pi^*_{C,\lambda_1}(w)\bigcap \Pi^*_{C,\lambda_2}(w)$ we note that for $i\in S$,

$$(\pi(j_i),\pi(j_i-1))=(\pi(i-1),\pi(i))=(\pi(j_i-1),\pi(j_i)),$$ which implies $\pi(i)=\pi(i-1)$. Now  $i$ is a generating vertex. But $\pi(i)=\pi(i-1)$ and hence is fixed, having chosen the first $i-1$  vertices. This lowers the order by a power of $n$ and hence the claim follows.
\vskip5pt

\noindent
(vii) Wigner and other matrices: Since the other cases such as Wigner and Toeplitz and Wigner and Reverse Circulant follow by similar and repetitive arguments we refrain from presenting a proof for them.
\subsection{Proof of Theorem~\ref{theorem:catalan}}
Let $w$ be a colored word of length $2k$ for a monomial $q=q(X,Y)$. Let $w'$ be
obtained from $w$ by a cyclic permutation, that is, there exists $l$ such that
$w'[i]=w[(i+l)\text{ mod}~ 2k]$. Note that $w'$ is a colored word for the monomial $q'$ obtained from $q$ by the same cyclic permutation. We have the following lemma.
\begin{lemma}\label{cyclic}
$|\Pi_C^*(w)|=|\Pi_C^*(w')|$ and $p_C(w)=p_C(w')$.
\end{lemma}
\begin{proof}[Proof of Lemma~\ref{cyclic}]
Let $\pi \in \Pi_C^*(w)$.
Let $\pi '(i)=\pi((i+l)\text{mod}~ 2k))$. Clearly, $\pi '(0)=\pi '(2k)$. Also
$$w'[i]=w'[j]\Rightarrow L^*(\pi'(i-1),\pi'(i))= L^*(\pi'(j-1),\pi'(j))$$
where $L^*$ is equal to $L_1$ or $L_2$ according as $w'[i]=w'[j]$ is an  $X$ match or a
$Y$ match. Hence, $\pi' \in \Pi_C^*(w').$

As $w$ can also be obtained from $w'$ by another cyclic permutation, it follows that the map $\pi \rightarrow \pi '$ is a bijection between $\Pi_C^*(w)$ and $\Pi_C^*(w')$. Hence $|\Pi_C^*(w)|=|\Pi_C^*(w')|$ and $p_C(w)=p_C(w')$.
\end{proof}
\begin{proof}[Proof of Theorem~\ref{theorem:catalan}](i) We  use induction on the length of the word.

If $k=1$ then $q=XX$ or $q=YY$. The only colored
Catalan word is $aa$ (drop superscript for ease). In either case,
$\pi(0)=i,\pi(1)=j,\pi(2)=i$ is a circuit in $\Pi_C^*(w)$ for$1\leq
i\leq n,1\leq j\leq n$. Hence, $|\Pi_C^*(w)|\geq n^{2}$ and the
result is true for $k=1$.

Now let us assume that the claim holds for all monomials $q$ of
length less than $2k$ and all Catalan words corresponding to $q$. By
Lemma \ref{cyclic}, without loss of generality we
assume that $w=aaw_1$ where $w_1$ is a Catalan word of length
$(2k-2)$. Now let $\pi '\in \Pi_C^*(w_1)$. For fixed $j$, $1\leq j
\leq n$, define $\pi$ by
\begin{align}
 \pi(0)&=\pi '(0) \label{ind1}\\
\pi(1)&=j \\
\pi(j)&=\pi '(j-2),\qquad j\geq 2\label{ind2}.
\end{align}
Clearly $\pi$ is a circuit and $\pi(0)=\pi(2)$ implies
$L(\pi(0),\pi(1))=L(\pi(1),\pi(2))$. Hence $\pi \in |\Pi_C^*(w)|$
and so, $|\Pi_C^*(w)|\geq n|\Pi_C^*(w_1)|\geq n^{k+1}$ and hence (i)  is proved.\vskip5pt

\noindent
(ii) We shall now show that $p_C(w)\leq 1$ for matrices
whose link functions satisfy \textit{Property B} and \textit{Property P}.
The proof is same as the proof of Theorem 2(ii) of \citet{bana:bose}, with appropriate changes to add color and index. We indicate the changes while keeping the notation similar to theirs for easy comparison.
The proof uses $(2k+1)$-tuple $\pi$ which are not necessarily circuit, that is, $\pi(0)=\pi(2k)$ is not assumed. Let $w$ be a colored Catalan word.
Define\begin{eqnarray*}
 C'(w)&=&\{\pi:  \ w[i]=w[j] \Rightarrow c_i=c_j \text{ and } L_{c_i}(\pi(i-1),\pi(i))=L_{c_j}(\pi(j-1),\pi(j))\}\\
 \Gamma_{i,j}(w)&=&\{\pi \in C(w): \pi(0)=i, \pi(2k)=j\}, \ (1 \le i,j \le n), \ \ \gamma_{i,j}(w)=|\Gamma_{i,j}(w)|.
\end{eqnarray*}
Clearly, $|\Pi_C^*(w)|=\sum_{i=1}^n\gamma_{i,i}(w)$. Now consider the  following statement
$\mathbf S_k'$ for all $k\geq 1$:
\vskip5pt

 $\mathbf S_k'$: For any colored Catalan $w$ of length $(2k)$, there exists $M_k>0$ such that
 $$\gamma_{i,j}(w) \le M_kn^{k-1} \ \ \text{for all} \ \  i \neq j\ \ \text{and}\ \
\frac{1}{n}\sum_{i=1}^n\left|\frac{\gamma_{i,i}(w)}{n^k}-1\right|=O(1/n).$$
 The proof of $\mathbf S_k'$ easily follows by repeating the steps of the proof of Theorem 2(ii) of \citet{bana:bose} and changing the  set $C(w)$  there by $C'(w)$ and using \textit{Property B}  and \textit{Property P}. To avoid repetitive arguments we skip the details.
Once the validity of $\mathbf S_k'$ is asserted, one gets
$p_C(w)\leq 1$ and the result now follows using part(i).
\end{proof}
\subsection{Proof of Theorem \ref{theorem:free1}}
We need the following development for describing freeness.

Let $S_n$ be the group of permutations of $(1, 2, \ldots n)$.
\begin{definition} \label{def:multilinear}Let $\mathcal A$ be an algebra. Let  $\psi_k:{\mathcal A}^k\longrightarrow C\,\  k>0$
be  multi linear functions. For $\alpha\in S_n$, let $c_1,c_2,\ldots
c_r$ be the cycles of $\alpha$. Then define
$$\psi_{\alpha}[A_1,A_2,\ldots,A_n]=\,\,\,\psi_{c_1}[A_1,A_2,\ldots,A_n]\psi_{c_2}[A_1,A_2,\ldots,A_n]\ldots \psi_{c_r}[A_1,A_2,\ldots,A_n]$$
where $$\psi_{c}[A_1,A_2,\ldots,A_n]=\psi_{p}\left(A_{i_1}A_{i_2}\ldots
A_{i_p}\right) \ \ \text{if}\ \ c=(i_1,i_2\ldots i_p).$$
\end{definition}
Freeness is intimately tied to non-crossing partitions.  We describe
the relevant portion of this relation in brief below. See Theorem
14.4 of \citet{spei} for more details. Let $NC_2(m)$ be the set of non-crossing pair-partitions of $\{1,2,\ldots,m\}$. A typical pair-partition $\pi$ will be written in the form $\{(r,\pi(r)), ~ r=1,2,\ldots,m\}$. For
$p=(p(1),p(2),\ldots,p(m))$ integers (also can be referred to as colors), let
$$NC_2^{(p)}(m)=\{\pi\in NC_2(m): p(\pi(r))=p(r) \text{ for all }
r=1,\ldots,m\}.$$ Suppose $d_1,d_2,\ldots, d_m,s_1,s_2,\ldots s_m$
are elements in some non-commutative probability space $(\mathcal
B,\varphi)$. Suppose $\{s_1,s_2,\ldots s_m\}$ are free and  each $s_i$
follows the semicircular law. Then the collections $\{s_1,s_2,\ldots s_m\}$ and
$\{d_1,d_2,\ldots d_m\}$ are free if and only if,
\begin{eqnarray}\label{free:description}\varphi(s_{p(1)}d_1\ldots s_{p(m)}d_m)&=&\sum_{\pi \in\,NC(m)}{k_{\pi}[s_{p(1)},\ldots s_{p(m)}].\,\varphi_{\pi\gamma}[d_1,\ldots, d_m]}\nonumber\\
&=&\sum_{\pi \in\,NC_2^{(p)}(m)}\varphi_{\pi\gamma}[d_1,\ldots , d_m],\end{eqnarray}
where $\gamma\in S_m$ is the cyclic permutation with one cycle and $\gamma=(1,2,\ldots,m-1,m)$. Here $k_n$ denotes the free cumulants and $k_\pi$ for a partition $\pi$ is defined along the same lines as  Definition~\ref{def:multilinear}.

We shall also drop the suffix $C$ from $p_C(w)$, $\Pi_C(w)$, $\Pi_C^*(w)$ etc. for simplicity. Fix a monomial $q$ of Wigner ($W$) and any other patterned matrix ($A$) of length $2k$.
To prove freeness we show that the limiting variables satisfy the relation~\eqref{free:description}. We have already remarked that freeness is intimately tied to non-crossing partitions but freeness in the limit can also be roughly described in terms of colored words in the following manner.
\begin{enumerate}
 \item If for a colored word the pair-partitions corresponding to the Wigner matrix cross, then $p(w)=0$.
\item If the pair-partition corresponding to the letters of matrix $A$ cross with any pair-partition of $W$ then also $p(w)=0$.
\end{enumerate}
For example, $p(w_1w_2w_1w_2a_1a_1)=0$ and $p(w_1a_1w_1a_1)=0$. This is essentially the main content of Lemma~\ref{lemma:strong} given below.

We will discuss in detail the proof of Theorem~\ref{theorem:free1} for $p=1$ and indicate how the results continue to hold for $p\geq 1$.

We need a few preliminary Lemmata to prove the main result. We first use these Lemmata to prove Theorem~\ref{theorem:free1} and then provide the proofs of the Lemmata.

 We now concentrate only on (colored) pair-matched words. For a word $w$ the pair $(i,j)\, 1\leq i<j\leq 2k$ is said to be a match if $w[i]=w[j]$. A match $(i,j)$ is said to be a $W$ match or an $A$ match according as $w[i]=w[j]$ is Wigner or $A$ letter.

 We define $w_{(i,j)}$ to be the word of length $j-i+1$ as $$w_{(i,j)}[k]=w[i-1+k] \text{ for all $1\leq k \leq j-i+1$.}$$

Let $w_{(i,j)^c}$  be the word of length $t+i-j-1$ obtained by
removing  $w_{(i,j)}$ from $w$, that is,
$$w_{(i,j)^c}[r]=\begin{cases}
w[r] & \text{if } r < i,\\
w[r+j-i+1] & \text{if } r \geq i.
\end{cases}$$
Note that in general these subwords may not be matched. If $(i,j)$ is a $W$ match, we will call $w_{(i,j)}$ a \textit{Wigner
string} of length $(j-i+1)$.  For instance, for the monomial
$WAAAAWWW$, $w=abbccadd$ is a word and $abbcca$ and $dd$ are Wigner
strings of length six and two respectively. For any word $w$,  we
define the following two classes:
\begin{eqnarray}\Pi^*_{(C2)}(w)&=&\{\pi \in\Pi^*(w): (i,j)~
W~\text{match} \Rightarrow (\pi(i-1),\pi(i))=(\pi(j),\pi(j-1)) \}, \\
\Pi^*_{(i,j)}(w)&=&\{\pi \in \Pi^*(w): (\pi(i-1),\pi(i))=
(\pi(j),\pi(j-1))\}.
\end{eqnarray}
Note that the condition appearing above involves $C2$ constraint defined in~\eqref{eq:c1:c2} and \begin{equation}\label{eq:ij}\Pi_{(C2)}^*(w)=\bigcap_{(i,j):\,W match}\Pi^*_{(i,j)}(w).\end{equation} It is well known that if we have a collection of only Wigner matrices then $p(w)\neq 0$ if and only if all the constraints in the word are $C2$ constraints. See for example~\citet{bosesen}. We need the following crucial extension in the present setup.
\begin{lemma}\label{lemma:strong} For a colored pair-matched word $w$ of length $2k$ with $p(w)\neq 0$ we
have:\vskip5pt

 \noindent (a) Every Wigner string is a colored pair-matched word; \vskip5pt

\noindent (b) For any $(i,j)$ which is a $W$ match we have
\begin{equation}\label{eq:red:wstring}
\lim_{n\longrightarrow\infty}\frac{|\Pi^*(w)-\Pi_{(i,j)}^*(w)|}{n^{1+k}}=0.
\end{equation}
\noindent
(c)\begin{equation}\label{eq:c2}\lim_{n\longrightarrow\infty}\frac{|\Pi^*(w)-\Pi_{(C2)}^*(w)|}{n^{1+k}}=0.\end{equation}
\end{lemma}
Note that $(c)$ and $(b)$ are equivalent by~\eqref{eq:ij} and as the number of pairs $(i,j)$ is finite.
\begin{lemma} \label{lemma:trace:product} Suppose  $X_n$ has LSD  and  they satisfy Assumption~\ref{A1} and~\ref{A2}, then for any  $l\ge 1$ and integers $(k_1,k_2,\ldots, k_l)$, we have
$$\E\left[\prod_{i=1}^{l}(\frac1n\Tr(\left(\frac{X_n}{\sqrt n}\right)^{k_i}))\right]-\prod_{i=1}^{l}{\E\left[\frac1n\Tr (\left(\frac{X_n}{\sqrt n}\right)^{k_i})\right]}
\to0 \text{ as } n\to\infty.$$
\end{lemma}
Assuming the above lemmas we now prove Theorem~\ref{theorem:free1}.
\begin{proof}[Proof of Theorem~\ref{theorem:free1}]
 We take a single copy of $W$ and $A$ to show the result but for multiple copies the proof essentially remains same modulo some notations.
Let $q$ be a typical monomial, $q=WA^{q(1)}WA^{q(2)}\ldots
WA^{q(m)}$ of length $2k$, where the $q(i)$'s may  equal $0$. So,
$k=m/2+(q(1)+q(2)+\ldots+q(m))/2$. From Theorem~\ref{main theorem-2}, for every such monomial $q$, $\frac1{n^{k+1}}\Tr(q)$ converges to say $\varphi(sa^{q(1)}\ldots sa^{q(m)})$, where $s$ follows the semicircular law and $a$ is the marginal limit of $A$, and $\varphi$ is the appropriate functional defined on the space of non-commutative polynomial algebra generated by $a$ and $s$. It is enough to prove that $\varphi$ satisfies~\eqref{free:description}.

Let us expand the expression for
\begin{align}
&\lim_{n\rightarrow\infty}\frac1{n^{1+k}}\E[\Tr(WA^{q(1)}WA^{q(2)}\ldots WA^{q(m)})] \nonumber&\\
&=\displaystyle{\lim_{n\rightarrow\infty}\frac{1}{n^{1+k}}\sum_{\substack{i(1),i(2),\ldots i(m)\\j(1),j(2),\ldots j(m)=1}}^{n}} \E[w_{i(1)j(1)}a^{q(1)}_{j(1)i(2)}w_{i(2)j(2)}a^{q(2)}_{j(2)i(3)}\ldots w_{i(m)j(m)}a^{q(m)}_{j(m)i(1)}] \label{eq:line1}&\\
&=\lim_{n\rightarrow\infty}\frac{1}{n^{1+k}}\sum_{w\in CW(2)}\sum_{\pi\in\Pi^*(w)}\E[\mathbb X_\pi] \nonumber&\\
&=\lim_{n\rightarrow\infty}\frac{1}{n^{1+k}} \sum_{w\in CW(2)}\sum_{\pi\in\Pi_{(C2)}^*(w)}\E[\mathbb X_\pi]\text{ (by Lemma~\ref{lemma:strong} (c) and Assumption~\eqref{A2})}.\label{eq:line2}
\end{align}
Colored pair-matched words of length $2k$ are in bijection with the set of pair-partitions on $\{1,2,\ldots,2k\}$ (denoted by $\mathcal P_2(2k)$).
 Now each such word $w$ induces $\sigma_{w}$ a pair-partition of  $\{1,2,\ldots m\}$ that is induced by only the Wigner matches i.e $(a,b)\,\in\,\sigma_w$ iff $(a,b)$ is a Wigner match. So given any  pair-partition $\sigma$ of $\{1,2,\ldots,m\}$, we denote by $[\sigma]_W$ the class of all $w$ which induce the partition $\sigma$. So the  sum in~\eqref{eq:line2} can be written as,
\begin{equation}\label{eq:line3}\lim_{n\to\infty}\frac{1}{n^{1+k}}\sum_{\sigma\in \mathcal P_2(m)}\sum_{w\in[\sigma]_W}\sum_{\pi\in\Pi_{(C2)}^*(w)}E[\mathbb X_\pi].\end{equation}
By $C2$ constraint imposed on the class $\Pi_{(C2)}^*(w)$, if $(r,s)$ is a $W$ match then $(i(r),j(r))=(j(s),i(s))$ (or, equivalently in terms of $\pi$ we have, $(\pi(r-1),\pi(r))=(\pi(s),\pi(s-1))$).

Therefore, we have the following string of equalities. Let $\tr$ be the normalized trace. The equality in~\eqref{eq:line4} follows from~\eqref{eq:line1} and~\eqref{eq:line2}. The steps in~\eqref{eq:line5},~\eqref{eq:line6} and~\eqref{eq:line7} follow easily from calculations similar to Proposition 22.32 of  \citet{spei}. The last step follows from the fact that the number of cycles of $\sigma\gamma$ is equal to $1+m/2$ if and only if $\sigma\in NC_2(m)$. The notation $\tr_{\sigma\gamma}$ is given in Definition~\ref{def:multilinear}.
\begin{align}
&\lim_{n\rightarrow\infty}\frac1{n^{k+1}}\E[\Tr(WA^{q(1)}WA^{q(2)}\ldots WA^{q(m)})]\nonumber&\\
&=\lim_{n\rightarrow\infty}\frac1{n^{k+1}}\sum_{\sigma\in \mathcal P_2(m)}\sum_{\substack{i(1),i(2),\ldots i(m)\\j(1),j(2),\ldots j(m)=1}}^n\prod_{(r,s)\in \sigma}\delta_{i(r)j(s)}\delta_{i(s)j(r)}\E[a^{q(1)}_{j(1)i(2)}\ldots a^{q(m)}_{j(m)i(1)}]\label{eq:line4}\\
&=\lim_{n\rightarrow\infty}\frac1{n^{k+1}}\sum_{\sigma\in \mathcal P_2(m)}\sum_{\substack{i(1),i(2),\ldots i(m)\\j(1),j(2),\ldots j(m)=1}}^n\prod_{(r,s)\in \sigma}\delta_{i(r)j(s)}\delta_{i(s)j(r)}\E[a^{q(1)}_{j(1)i(\gamma(1))}\ldots a^{q(m)}_{j(m)i(\gamma(m))}]\label{eq:line5}\\
&=\lim_{n\rightarrow\infty}\frac1{n^{k+1}}\sum_{\sigma\in \mathcal P_2(m)}\sum_{\substack{i(1),i(2),\ldots i(m)\\j(1),j(2),\ldots j(m)=1}}^n\prod_{r=1}^m\delta_{i(r)j(\sigma(r))}\E[a^{q(1)}_{j(1)i(\gamma(1))}\ldots a^{q(m)}_{j(m)i(\gamma(m))}]\label{eq:line6}\\
&=\lim_{n\rightarrow\infty}\frac1{n^{k+1}}\sum_{\sigma\in \mathcal P_2(m)}\sum_{j(1),j(2),\ldots j(m)=1}^n\E[a^{q(1)}_{j(1)j(\sigma\gamma(1))}\ldots a^{q(m)}_{j(m)j(\sigma\gamma(m))}]\label{eq:line7}\\
&=\sum_{\sigma\in NC_2(m)}\lim_{n\rightarrow\infty}\E\left(\tr_{\sigma\gamma}[A^{(q_1)},A^{(q_2)},\ldots ,A^{(q_m)})]\right). \nonumber
\end{align}
 Now it follows from Lemma~\ref{lemma:trace:product} that,
\begin{align*}\sum_{\sigma\in NC_2(m)}\lim_{n\rightarrow\infty}\E\left(\tr_{\sigma\gamma}[A^{(q_1)},A^{(q_2)},\ldots ,A^{(q_m)})]\right)&=\sum_{\sigma\in NC_2(m)}\lim_{n\rightarrow\infty}(\E\tr)_{\sigma\gamma}[A^{(q_1)},A^{(q_2)},\ldots ,A^{(q_m)})]\\
&=\sum_{\sigma\in NC_2(m)}\varphi_{\sigma\gamma}[a^{(q_1)},a^{(q_2)},\ldots ,a^{(q_m)})].\\
\end{align*}
This shows~\ref{free:description} and hence freeness in the limit.

The above method can be easily extended to plug in more independent copies of $W$ and $A$. The following details will be necessary.
\begin{enumerate}
 \item The extension of Lemmata~\ref{lemma:strong} and~\ref{lemma:trace:product}. Note that these extensions can be easily obtained using the injective mapping $\psi$ described in Section~\ref{main results} and used in Theorem~\ref{main theorem-1}.
\item When we consider several independent copies of the Wigner matrix the product in~\eqref{eq:line6} gets replaced by $$\prod_{r=1}^m\delta_{i(r)j(\sigma(r))}\delta_{p(r)p(\sigma(r))}.$$ Here $(p(1),p(2),\ldots,p(m))$ denotes the colors corresponding to the independent Wigner matrices. The calculations are similar to Theorem 22.35 of~\citet{spei}.
\end{enumerate}
The rest are some algebraic details, which we skip.
\end{proof}
Having proved the Theorem we now come back to the proof of Lemma \ref{lemma:strong} and \ref{lemma:trace:product}.
%
The next Lemma turns out to be the most essential ingredient in proving Lemma~\ref{lemma:strong} and it points out the behavior of a colored pair-matched word which contains a Wigner string inside it.
\begin{lemma}\label{lemma:C2}For any colored pair-matched word $w$ and a Wigner string $w_{(i,j)}$ which is a pair-matched word and satisfies equation~\eqref{eq:red:wstring}
\begin{equation}\label{eq:lemma:pw}p(w)=p(w_{(i,j)})p(w_{(i,j)^c}).\end{equation}
Further, if $w_{(i+1,j-1)}$ and $w_{(i,j)^c}$ satisfy~\eqref{eq:c2} then so does $w$.
\end{lemma}
\begin{proof}
Given any $\pi_1\in \Pi^*(w_{(i+1,j-1)})$ and $\pi_2\in \Pi^*(w_{(i,j)^c})$ construct $\pi$ as:
$$\pi=(\pi_2(0),\ldots \pi_2(i-1),\pi_1(0),\ldots \pi_1(j-i-1)=\pi_1(0),\pi_2(i-1),\ldots(2k-j+i-1)) \in \Pi_{(i,j)}^*(w).$$
Conversely,  from any $\pi \in\Pi_{(i,j)}^*(w)$ one can construct $\pi_1$ and $\pi_2$  by reversing the above construction.

So we have \begin{equation}\label{eq:prod}|\Pi_{(i,j)}^*(w)|=|\Pi^*(w_{(i+1,j-1)})||\Pi^*(w_{(i,j)^c})|.\end{equation} Let $|w_{(i+1,j-1)}|=2l_1$ and $|w_{(i,j)^c}|=2l_2$ and note that $(1+l_1)+(1+l_2)=k+1$.

Now using the fact that $w_{(i,j)}$ satisfies~\eqref{eq:red:wstring} and dividing equation~\eqref{eq:prod} by $n^{k+1}$ we get in the limit, $$p(w)=p(w_{(i+1,j-1)})p(w_{(i,j)}^c).$$ Now we claim that \begin{equation}\label{eq:preduced}|\Pi^*(w_{(i,j)})|= n|\Pi^*(w_{(i+1,j-1)})|.\end{equation}
Now given  $\pi\in \Pi^*(w_{(i,j)})$, one can always get a $\pi'\in \Pi^*(w_{(i+1,j-1)})$, where the $\pi(i-1)$ is arbitrary and hence $\frac{|\Pi^*(w_{(i,j)})|}{n}\leq |\Pi^*(w_{(i+1,j-1)})|$. Also given a $\pi'\in \Pi^*(w_{(i+1,j-1)})$ one can choose $\pi(i-1)$ in $n$ ways and also assign $\pi(j)=\pi(i-1)$ or $\pi(i)$, making $j$ a dependent vertex. So we get that, $ |\Pi^*(w_{(i,j)})|\geq n|\Pi^*(w_{(i+1,j-1)})|$. This shows~\eqref{eq:preduced}. So from~\eqref{eq:preduced} it follows that $$p(w_{(i,j)})=p(w_{(i+1,j-1)}),$$ whenever $w_{(i,j)}$  is a Wigner string.

 Also note that from the first construction, $$|\Pi_{(C2)}^*(w)|=|\Pi_{(C2)}^*(w_{(i+1,j-1)})||\Pi_{(C2)}^*(w_{(i,j)^c})|.$$ Now suppose $w_{(i+1,j-1)}$ and $w_{(i,j)^c}$ satisfy~\eqref{eq:c2}. So we have that $$|\Pi^*(w_{(i+1,j-1)})|=|\Pi_{(C2)}^*(w_{(i+1,j-1)})|+\text{o}(n^{l_1+1}) \text
{ and } |\Pi^*(w_{(i,j)^c})|=|\Pi_{(C2)}^*(w_{(i,j)^c})|+\text{o}(n^{l_2+1}).$$ Multiplying these and using the fact (from~\eqref{eq:lemma:pw}) $|\Pi^*(w)|=|\Pi^*(w_{(i+1,j-1)})||\Pi^*(w_{(i,j)^c})|+\text{o}(n^{k+1})$, the result follows.
\end{proof}
We now give a proof of Lemma~\ref{lemma:strong}.
\begin{proof}[Proof of Lemma~\ref{lemma:strong}]
We use induction on the length $l$ of the Wigner string. Let $w$ be a pair-matched colored word of length $2k$ with $p(w)\neq 0$. First suppose the Wigner string is of length $2$, that is, $l=2$. We may without loss of generality assume them in the starting position. So we for any $\pi\in \Pi^*(w)$ with above property we have $$(\pi(0),\pi(1))=\begin{cases}
                                                                                                                                                                                                                                                                                                                                                        (\pi(1),\pi(2))\\
(\pi(2),\pi(1)).                                                                                                                                                                                                                                                                                                                                                         \end{cases}$$
In the first case $\pi(0)=\pi(1)=\pi(2)$ and so $\pi(1)$ is not generating vertex and this lowers the number of generating vertices (which is not possible as $p(w)\neq 0$). Hence, the only possibility is $(\pi(0),\pi(1))=(\pi(2),\pi(1))$ and the circuit is complete for the Wigner string and so it is a pair-matched word, proving part (a). Also, as a result of the above arguments only $C2$ constraints survive, which shows $(b)$.

Now suppose the result holds for all Wigner strings of length strictly less than $l$. Consider a Wigner string of length $l$, say $w_{(1,l)}$ (we assume it to start from the first position). We break the proof into two cases I and II. In case I, we suppose that the Wigner string has a Wigner string of smaller order and use induction hypothesis and Lemma~\ref{lemma:C2} to show the result. In Case II, we assume that there is no Wigner string inside. So there is a string of letters coming from matrix $A$ after a Wigner letter. We show that this string is pair-matched and the last Wigner letter before the $l$-th position is essentially at the first position. This also implies that the string within a Wigner string do not cross a Wigner letter.
\vskip5pt

{\bf Case I:} Suppose that $w_{(1,l)}$ contains a Wigner string of length less than $l$ at the position $(p,q)$ with $1\le p<q\le l$. Since $w_{(p,q)}$ is a Wigner string, by Lemma~\ref{lemma:C2} we have,
$$p(w)=p(w_{(p,q)})p(w_{(p,q)^c})\neq 0.$$
So by induction hypothesis and the fact that both $p(w_{(p,q)})$ and $p(w_{(p,q)^c})$ are not equal to zero we have, $w_{(p,q)}$ and $w_{(p,q)^c}$ are pair-matched words and they also satisfy~\eqref{eq:red:wstring}. So $w_{(1,l)}$ is a pair-matched word, as it is made up of $w_{(p,q)}$ and $w_{(p,q)^c}$ which are pair-matched. Also from second part of Lemma~\ref{lemma:C2}, we have $w_{(1,l)}$ satisfies part $(b)$ and $(c)$.

\vskip5pt
{\bf Case II:} Suppose there is no Wigner string in the first $l$ positions.Consider
 the last Wigner letter in the first $l-1$ positions, say at
 position $j_0$. Since
 there is no Wigner string of smaller length, $\pi(j_0)$ is a generating vertex. Also, as $j_0$ is the last Wigner letter, the positions from $j_0$ to $l-1$ are all letters coming from the matrix $A$.

Now we use the structure of the matrix $A$.
\vskip5pt

{\bf Subcase II(i):} Suppose $A$ is  a Toeplitz matrix. Let  $s_i=(\pi(j_0+i)-\pi(j_0+i-1))$ with $i=1,2,\ldots, l-1-j_0$. Now consider the following equation
\begin{equation}
\label{equation:Toeplitz} s_1 + s_2 \ldots + s_{l-1-j_0}=(\pi(l-1)-\pi(j_0)).
\end{equation}
If for any $j$, $w[j]$ is the first appearance of that letter, then consider $s_j$ to be an independent variable (can be chosen freely). Then due to the Toeplitz link function, if $w[k]=w[j]$, where $k>j$, then $s_k=\pm s_j$.
 Since $(1,l)$ is a W match, $\pi(l-1)$ is either $\pi(0)$ or $\pi(1)$ and hence $\pi(l-1)$ is not a generating vertex. Note that~\eqref{equation:Toeplitz} is a constraint on the independent variables unless $s_1+\ldots+s_{l-1-j_0}=0$. If this is non-zero, this non-trivial constraint lowers the number of independent variables and hence the limit contribution will be zero, which is not possible as $p(w)\neq0$. So we must have,
$$\pi(l-1)=\pi(j_0) \qquad \text{and}\qquad j_0=1.$$ This also shows $(\pi(l),\pi(l-1))=(\pi(0),\pi(1))$ and hence $w_{(1,l)}$ is a colored word. As $s_1+\ldots+s_{l-1-j_0}=0$, all the independent variables occur twice with different signs in the left side,
since otherwise it would again mean a non-trivial relation among  them and thus would lower the order. Hence we conclude that the Toeplitz letters inside the first $l$ positions are also pair-matched. Since the $C2$ constraint is satisfied at the position $(1,l)$, part $(b)$ also holds.\vskip5pt

\textbf{Subcase II(ii):} Suppose $A$ is a Hankel matrix.
We write, $t_i=(\pi(j_0+i)+\pi(j_0+i-1))$ and consider
\begin{equation}\label{equation:Hankel}  -t_1+t_2-t_3\ldots(-1)^{l-j_0-1}t_{l-j_0-1}=(-1)^{l-j_0-1}(\pi(l-1)-\pi(j_0)).
\end{equation}
Now again as earlier, the $t_i$'s are independent variables, and so this implies that again to avoid a non-trivial constraint which would lower the order, both sides of the equation(\ref{equation:Hankel}) have to vanish, which automatically leads to the conclusion that $\pi(l-1)=\pi(j_0)=\pi(1)$.
So $j_0=1$  and again the Wigner paired string of length $l$ is pair-matched. Part $(b)$ also follows as the $C2$ constraint holds.

\textbf{Subcase II(iii):} $A$ is Symmetric or Reverse Circulant. Note that they have link functions which are quite similar to Toeplitz and Hankel respectively, the proofs are very similar to the above two cases and hence we skip them.
\end{proof}
\begin{proof}[Proof of Lemma~\ref{lemma:trace:product}]
We first show that, \begin{equation}\label{equation:suff}\E\left[\prod_{i=1}^{l}\left(\tr\frac{X_n^{k_i}}{n^{k_i/2}}-\E\left[\tr \frac{X_n^{k_i}}{{n}^{k_i/2}}\right]\right)\right]=\text{O}(\frac{1}{n}) \text{ as }n\to\infty,\end{equation}
where $\tr$ denotes the normalized trace.
To prove \eqref{equation:suff}, we see that,
\begin{equation}\label{eq:prod:show}\E\left[\prod_{i=1}^{l}\left(\tr\frac{X_n^{k_i}}{n^{k_i/2}}-\E\left[\tr \frac{X_n^{k_i}}{{n}^{k_i/2}}\right]\right)\right]=\frac{1}{n^{\sum_{i=1}^l k_i/2+l}}\sum_{\pi_1,\pi_2,..\pi_l}\E[(\prod_{j=1}^{l}(X_{\pi_i}-\E(X_{\pi_i})))].\end{equation}
If the circuit $\pi_i$ is not jointly matched with the other circuits then $\E{X_{\pi_i}}=0$ and
$$\E[(\prod_{j=1}^{l}(X_{\pi_i}-\E(X_{\pi_i})))]= \E[X_{\pi_i}(\prod_{j \neq i}(X_{\pi_i}-\E(X_{\pi_i})))]=0.$$
 If any of the circuits is self matched i.e. it has no cross matched edge then
 $$\E[(\prod_{j=1}^{l}(X_{\pi_i}-\E(X_{\pi_i})))]= \E[X_{\pi_i}-\E(X_{\pi_i})]\E[(\prod_{j \neq i}(X_{\pi_i}-\E(X_{\pi_i})))]=0.$$
Now total number of circuits  $\{\pi_1,\pi_2,\ldots \pi_l\}$ where each edge appears at least twice and one edge at least thrice is $\leq C n^{\sum_{i=1}^l{k_i/2}+l-1}$, by Property B. Hence using Assumption~\eqref{A2} such terms in~\eqref{eq:prod:show}  are of the order $\text{O}(\frac1n)$. Now consider rest of terms where all the edges appear exactly twice. As a consequence $\sum_{i=1}^lk_i$ is even. Also number of partitions of $\frac12\sum_{i=1}^l{k_i}$ into $l$ circuits is  independent of $n$.
We need to consider only $\{\pi_1,\pi_2,\ldots \pi_l\}$ which are jointly matched but not self matched.

If we  prove that for such a partition the number of circuits is less than  $Cn^{\sum_{i=1}^l{k_i}+l-1}$ we are done since the number of such partitions is independent of $n$ and~\eqref{eq:boundedexp}.

Since $\pi_1$ is not self matched we can without loss of generality assume that the edge value for $(\pi(0),\pi(1))$ occurs exactly once in $\pi_1$. So construct $\pi_1$ as follows. First choose  $\pi_1(0)=\pi_1(k_{1})$ and then choose the remaining vertices in the order ${\pi_1(k_{1}),\pi_1(k_{1}-1)\ldots\pi_1(1)}$. One sees that we loose one degree of freedom as in this way the edge $(\pi(0),\pi(1))$ is determined and we cannot choose it arbitrarily.

The result now follows from~\eqref{equation:suff} by using induction.  For $l=2$ expanding and using the fact that expected normalized trace of the powers of $X_n/\sqrt n$ converges we get,
\begin{align*}
&  \E\left[\prod_{i=1}^{2}\left(\tr\frac{X_n^{k_i}}{n^{k_i/2}}-\E\left[\tr \frac{X_n^{k_i}}{{n}^{k_i/2}}\right]\right)\right]&\\
&= \E\left[\left(\tr\frac{X_n^{k_1}}{n^{k_1/2}}-\E\left[\tr \frac{X_n^{k_1}}{{n}^{k_1/2}}\right]\right)\left(\tr\frac{X_n^{k_2}}{n^{k_2/2}}-\E\left[\tr \frac{X_n^{k_2}}{{n}^{k_2/2}}\right]\right)\right]\\
&=\E\left[\tr \frac{X_n^{k_1}}{ n^{k_1/2}}\tr \frac{X_n^{k_2}}{n^{k_2/2}}\right]-\E\left[\tr\frac{X_n^{k_1}}{n^{k_1}}\right]\E\left[\tr\frac{X_n^{k_2}}{n^{k_2}}\right]\to0 \text{ as }n\to\infty.\\
\end{align*}
So the result holds for $l=2$. Now suppose it is true for all $2\le m< l$. We expand $$\lim_{n\rightarrow\infty} \E[\prod_{i=1}^{l}(\tr((\frac{X_n}{\sqrt n})^{k_i})-\E(\tr((\frac{X_n}{\sqrt n})^{k_i})))]=0$$ to get

$$\lim_{n\rightarrow\infty}\sum_{m=1}^l(-1)^m\sum_{i_1<i_2\ldots<i_m}\E[\prod_{j=1}^m \tr({(\frac{X_n}{\sqrt n})^{k_{i_j}}})]\prod_{i\notin\{i_1,i_2,\ldots  i_m\}}\E[\tr((\frac{X_n}{\sqrt n})^{k_i})]=0.$$

Now using the result for products of smaller order successively,

$$\lim_{n\rightarrow\infty}(-1)^l \E[\prod_{j=1}^l \tr({(\frac{X_n}{\sqrt n})^{k_{j}}})]=
\lim_{n\rightarrow\infty}\sum_{m <l}(-1)^m\sum_{i_1<i_2\ldots <i_m}\E[\prod_{j=1}^m \tr({(\frac{X_n}{\sqrt n})^{k_{i_j}}})]\prod_{i \notin \{i_1,i_2,\ldots  i_m\}}\E[\tr((\frac{X_n}{\sqrt n})^{k_i})].$$
Now  every term in right side is by induction hypothesis $\lim_{n\rightarrow\infty}\prod_{i=1}^l\E[\tr((\frac{X_n}{\sqrt n})^{k_i})]$ and from this the Lemma follows.
\end{proof}
\begin{proof}[Proof Remark~\ref{rem:gue:free}]
We just briefly sketch the arguments as the proof is quite similar to the previous section but much easier.
Note that if $W$ is centered GUE with variance $1/n$ then,
\begin{equation}\label{eq:var:gue}
\E[W_{ij}W_{kl}]=\frac1n \delta_{il}\delta_{jk}.
\end{equation}
 This equation~\eqref{eq:var:gue} provides the $C2$ constraint in the proof of Theorem~\ref{theorem:free1}. So following the steps in the proof of Theorem~\ref{theorem:free1} we have
 \begin{align*}
&\lim_{n\rightarrow\infty}\frac1{n^{k+1}}\E[\Tr(WA^{q(1)}WA^{q(2)}\ldots WA^{q(m)})]\nonumber&\\
&=\lim_{n\rightarrow\infty}\frac1{n^{k+1}}\sum_{\sigma\in \mathcal P_2(m)}\sum_{\substack{i(1),i(2),\ldots i(m)\\j(1),j(2),\ldots j(m)=1}}^n\prod_{(r,s)\in \sigma}\delta_{i(r)j(s)}\delta_{i(s)j(r)}\E[a^{q(1)}_{j(1)i(2)}\ldots a^{q(m)}_{j(m)i(1)}]\label{eq:line4}\\
&=\sum_{\sigma\in NC_2(m)}\lim_{n\rightarrow\infty}\E\left(\tr_{\sigma\gamma}[A^{(q_1)},A^{(q_2)},\ldots ,A^{(q_m)})]\right). \nonumber
\end{align*}
Now the result follows by applying Lemma~\ref{lemma:trace:product} which holds under Property B and existence of LSD.
\end{proof}
\section*{Acknowledgement}
The authors  thank  Wlodek Bryc,  Alice Guionnet, Koushik Saha and Roland Speicher for helpful comments and for providing some important references. They are also grateful to the Referee for his/her comments which have led to a much better presentation.


\end{document}